\documentclass[12pt]{amsart}
\usepackage{amssymb,amsmath,amsfonts,latexsym}
\usepackage{bm}

\newcommand{\newsection}[1]{\setcounter{equation}{0} \section{#1}}
\newcommand{\bea}{\begin{eqnarray}}
	\newcommand{\eea}{\end{eqnarray}}

\newcommand{\clb}{\mathcal{B}}

\newcommand{\cld}{\mathcal{D}}
\newcommand{\cle}{\mathcal{E}}
\newcommand{\clf}{\mathcal{F}}
\newcommand{\clg}{\mathcal{G}}
\newcommand{\clh}{\mathcal{H}}
\newcommand{\clk}{\mathcal{K}}
\newcommand{\cll}{\mathcal{L}}
\newcommand{\clm}{\mathcal{M}}
\newcommand{\cln}{\mathcal{N}}

\newcommand{\clp}{\mathcal{P}}
\newcommand{\clq}{\mathcal{Q}}
\newcommand{\clr}{\mathcal{R}}
\newcommand{\cls}{\mathcal{S}}

\newcommand{\clw}{\mathcal{W}}

\newcommand{\D}{\mathbb{D}}
\newcommand{\bS}{\mathbb{S}}

\newcommand{\raro}{\rightarrow}

\def\textmatrix#1&#2\\#3&#4\\{\bigl({#1 \atop #3}\ {#2 \atop #4}\bigr)}
\def\dispmatrix#1&#2\\#3&#4\\{\left({#1 \atop #3}\ {#2 \atop #4}\right)}
\newcommand{\be}{\begin{equation}}
	\newcommand{\ee}{\end{equation}}
\newcommand{\ben}{\begin{eqnarray*}}
	\newcommand{\een}{\end{eqnarray*}}

\newcommand{\NI}{\noindent}

\newcommand{\bi}{\begin{itemize}}
	\newcommand{\ei}{\end{itemize}}

\newcommand{\Z}{\mathbb{Z}}

\newcommand\la{{\langle }}
\newcommand\ra{{\rangle}}

\theoremstyle{definition}

\theoremstyle{plain}

\newtheorem{thm}{Theorem}[section]
\newtheorem{cor}[thm]{Corollary}
\newtheorem{lem}[thm]{Lemma}
\newtheorem{prop}[thm]{Proposition}
\theoremstyle{definition}
\newtheorem{defn}[thm]{Definition}

\newtheorem{question}[thm]{Question}

\numberwithin{equation}{section}

\let\phi=\varphi

\begin{document}
	
\title[Inner and characteristic functions in polydiscs]{Inner and characteristic functions in polydiscs}

\author[Debnath]{Ramlal Debnath}
\address{Department of Mathematics, KTH Royal Institute of Technology, Stockholm,  Sweden}
\email{ramlal@kth.se, ramlaldebnath@gmail.com}

\author[Pradhan]{Deepak K. Pradhan}
\address{Department of Mathematics, Indian Institute of Technology Hyderabad, Kandi, Sangareddy, Telangana, 502284, India}
\email{deepak12pradhan@gmail.com, dkpradhan@math.iith.ac.in}
	
\author[Sarkar]{Jaydeb Sarkar}
\address{Indian Statistical Institute, Statistics and Mathematics Unit, 8th Mile, Mysore Road, Bangalore, 560059, India}
\email{jay@isibang.ac.in, jaydeb@gmail.com}

	


\begin{abstract}
Characteristic functions of linear operators are analytic functions that serve as complete unitary invariants. Such functions, as long as they are built in a natural and canonical manner, provide representations of inner functions on a suitable domain and make significant contributions to the development of various theories in Hilbert function spaces. In this paper, we solve this problem in polydiscs. In particular, we present a concrete description of the characteristic functions of tuples of commuting pure contractions and, consequently, provide a description of inner functions on polydiscs.
\end{abstract}
	
\maketitle

\tableofcontents

\newsection{Introduction}\label{sec: intro}

Traditionally, there is a deep interconnection between bounded analytic functions and bounded linear operators. An example of this is representations of inner functions, representations of bounded linear operators, invariant subspace problem, characteristic functions, and unitary invariants; as a whole, this package is classically well-known as model theory \cite{NF, Nik}. We note that these links are well-established, fruitful, and extensively studied, although they remain far from fully understood. The notion of characteristic functions serves as one of the key connectors to this relationship and is central to the well-known theory of Sz.-Nagy and Foias \cite{NF}.

The concept of characteristic functions for commuting tuples of contractions has long been a challenging and unresolved issue. A concrete solution to this problem could provide representations of inner functions on $\D^n$, where $\D = \{z \in \mathbb{C}: |z| < 1\}$ and
\[
\D^n = \underbrace{\D \times \cdots \times \D}_{n-\text{times}},
\]
is the open unit polydisc in $\mathbb{C}^n$. Representations of inner functions on $\D^n$, $n > 1$, is another widely recognized difficult problem. This interdependence further underscores the complexity of the issue. The purpose of this paper is to address these aspects of multivariable operator theory and function theory on $\D^n$, $n > 1$. We present an explicit formula for characteristic functions that serve as complete unitary invariants for commuting tuples of contractions. The constructions in particular yield the first-ever representations of inner functions defined on $\D^n$, $n > 1$. 

A common approach to attempt Hilbert function theoretic problems on $\D^n$ is the well-known technique of fractional linear transformations (FLT), or transfer functions, which is further refined by the theory of Schur-Agler class functions (see Agler \cite{Ag 1, Ag 2}, and also see \cite{BLTT}). However, time has shown that the FLT, or the Schur-Agler class function approach, may not always be as effective as it appears in single operators and function theory on $\D$. A significant example of this limitation is the recent complete solution \cite{KD-S} to the Nevanlinna-Pick interpolation problem on $\D^n$, $n > 1$, which offers an entirely new perspective on functions and operators on $\D^n$, independent of Schur-Agler class techniques. In this paper, we also tackle the issue of inner functions and characteristic functions by presenting a new theory that differs from the conventional FLT framework.

\begin{question}\label{quest}
Here, one seeks solutions to the following three basic problems:
\begin{enumerate}
\item Identify the class of $n$-tuples that admit characteristic functions.
\item And then compute the characteristic functions for those tuples. 
\item Finally, use (1) and (2) to represent inner functions on $\D^n$, $n \geq 1$. 
\end{enumerate}

\end{question}

Let us proceed to explain the solutions to these (and many more related) problems. Given Hilbert spaces $\cle$ and $\cle_*$, designate $\clb(\cle, \cle_*)$ (or simply $\clb(\cle)$ if $\cle_* = \cle$) as the space of all bounded linear operators from $\cle$ to $\cle_*$. Here, all the Hilbert spaces are separable and over $\mathbb{C}$. Designate $H^\infty_{\clb(\cle, \cle_*)}(\D^n)$ as the space of all bounded $\clb(\cle, \cle_*)$-valued analytic functions on $\D^n$. We refer to a function $\Theta\in H^{\infty}_{\clb(\cle,\cle_*)}(\D^n)$ as \textit{inner} if
\[
\Theta(z)^*\Theta(z)=I_{\cle},
\]
for almost every $z\in \mathbb{T}^n$ (where $\mathbb{T}^n$ denotes the distinguished boundary of $\D^n$). We also denote the $\cle_*$-valued Hardy space over $\D^n$ as $H^2_{\cle_*}(\D^n)$ and write the $n$-tuple of multiplication operators by the coordinate functions acting on $H^2_{\cle_*}(\D^n)$ as $(M_{z_1}, \ldots, M_{z_n})$ (see Section \ref{sec: basics} for more details). Throughout the paper, we follow the notation
\[
I_n = \{1, \ldots, n\}.
\]
We will primarily assume that $n \geq 1$, and it will be clear from the context if the $n$ does not include $1$. A closed subspace $\clq$ of $H^2_{\cle_*}(\D^n)$ is said to be a \textit{model space} or a \textit{quotient module} if $M_{z_i}^* \clq \subseteq \clq$ for all $i \in I_n$. For each $i \in I_n$, define the \textit{model operator}
\[
C_i = P_\clq M_{z_i}|_\clq \in \clb(\clq).
\]
Throughout the paper, given a closed subspace $\cls$ of a Hilbert space $\clh$, we denote the orthogonal projection of $\clh$ onto $\cls$ by $P_\cls$. Recall that $H^2_{\cle_*}(\D^n)$ is a reproducing kernel Hilbert space corresponding to the kernel function
\[
\D^n\times \D^n \ni  (z, w) \mapsto  \mathbb{S}_n(z, w)I_{\cle_*},
\]
where $\mathbb{S}_n$ is the \textit{Szeg\"{o} kernel} of $\D^n$. Recall that
\[
\bS_n(z, w) = \prod_{i \in I_n} (1 - z_i \bar{w}_i)^{-1},
\]
for all $z, w \in \D^n$. Throughout the paper, the letters $z = (z_1, \ldots, z_n)$ and $w = (w_1, \ldots, w_n)$ will represent the coordinate functions on $\mathbb{C}^n$.

Representations of quotient modules, when $n > 1$, are by far more complicated. However, function theory directly connects to the construction of certain quotient modules in a natural way: Given an inner function $\Theta\in H^{\infty}_{\clb(\cle,\cle_*)}(\D^n)$, consider the quotient module
\[
\clq_{\Theta}:=H^2_{\cle_*}(\D^n)/ \Theta H^2_{\cle}(\D^n).
\]
We refer to these inner function-based quotient modules as \textit{Beurling quotient modules} (see Definition \ref{def: beurling QM}). Note that when $n=1$, all quotient modules are Beurling quotient modules— a celebrated result attributed to Beurling, Lax, and Halmos \cite[Chap. V, Theorem 3.3]{NF}. The inner function $\Theta$ can be regarded as the characteristic function of the tuple $C:= (C_1, \ldots, C_n)$. In other words, the $n$-tuple $C$ on $\clq_\Theta$ admits characteristic functions (namely, $\Theta$ itself). It is now natural to identify $n$-tuples that admit characteristic functions (in terms of inner functions as mentioned above). This is the problem raised in part 1 of Question \ref{quest}. The Szeg\"{o} tuples are the half way answer to this question: Given a Hilbert space $\clh$, we let
\[
\clb_c^n(\clh) = \{T = (T_1, \ldots, T_n): T_i T_j = T_j T_i, \text{ and } \|T_i\| \leq 1 \text{ for all } i,j \in I_n\},
\]
the set of all $n$-tuples of commuting contractions on $\clh$. We set
\[
\mathbb{S}^{-1}_n(T, T^*) := \sum_{k \in \Z^n_+, |k| \leq n} (-1)^{|k|}T^k T^{*k},
\]
where $T^k = T_1^{k_1} \cdots T_n^{k_n}$, $T^{*k} = T_1^{*k_1} \cdots T_n^{*k_n}$, and $|k| = \sum_{i=1}^{n} k_i$ for all $k = (k_1, \ldots, k_n) \in \Z_+^n$.

\begin{defn}\label{def: Szego tuples}
$T \in \clb_c^n(\clh)$ is said to be a Szeg\"{o} tuple if $\mathbb{S}^{-1}_n(T, T^*) \geq 0$ and $T_i \in C_{\cdot 0}$ for all $i\in I_n$. The set of all Szeg\"{o} $n$-tuples acting on $\clh$ will be denoted by $\bS_n(\clh)$, that is,
\[
\bS_n(\clh) = \{T \in \clb_c^n(\clh): \mathbb{S}^{-1}_n(T, T^*) \geq 0 \text{ and } T_i \in C_{\cdot 0}, i=1, \ldots, n\}.
\]
\end{defn}

Recall that a contraction $X\in \clb(\clh)$ is in $C_{\cdot 0}$ if $SOT-\lim_{m\raro \infty} X^{*m} = 0$ (in short, $X \in C_{\cdot 0}$). The above positivity conditions of tuples align with the concept of the Szeg\"{o} kernel on the polydisc (refer to Section \ref{sec: basics} for further details).

The Szeg\"{o} kernel and Szeg\"{o} tuples are interconnected in the following manner: Let $T \in \bS_n(\clh)$. Then $T$ is jointly unitarily equivalent to the tuple of model operators $C = (C_1, \ldots, C_n)$ on some quotient space $\clq$ (see Theorem \ref{Dilation Thm} and \eqref{eqn: T and CT}). Two tuples $X = (X_1, \ldots, X_n)$ on $\clk$ and $Y = (Y_1, \ldots, Y_n)$ on $\cll$ are said to be \textit{jointly unitarily equivalent} if there exists a unitary operator $U: \clk \raro \cll$ such that
\[
U X_i U^* = Y_i \qquad (i \in I_n).
\]
We denote this by $X \cong Y$. Studying Szeg\"{o} tuples is thus equivalent to studying tuples of model operators. This result is due to M\"{u}ller and Vasilescu \cite{MV} (also see Timotin \cite{Timotin}). Therefore, if a tuple admits a characteristic function, then it is necessarily a Szeg\"{o} tuple. It is well-known that this is not a sufficient condition, and the following is the missing information that is needed for a Szeg\"{o} tuple to admit a characteristic function: Let $T \in \bS_n(\clh)$. Then $T$ admits a characteristic function if and only if
\[
(I - T_i^* T_i) (I - T_j^* T_j) = 0,
\]
for all $i\neq j$. This was proved in \cite[Theorem 3.2]{Beurling quotient module}. In this paper, we include one more condition that is more effective for the present consideration (See Theorem \ref{Beurling_qm}):

\begin{thm}
Let $T \in \bS_n(\clh)$. Then $T$ admits a characteristic function if and only if $T_i|_{\cld_{T_j}}: \cld_{T_j} \raro \cld_{T_j}$ is an isometry for all $i\neq j$.
\end{thm}

Recall the standard notation that for a contraction $X \in \clb(\clh)$, one denotes by
\begin{equation}\label{eqn: defect}
D_X = (I_{\clh}-X^*X)^{1/2},
\end{equation}
the defect operator and $\cld_X = \overline{ran}D_X$ the corresponding defect space. Since $X^*$ is also a contraction, we also have $D_{X^*} = (I_{\clh}-X X^*)^{1/2}$ and $\cld_{X^*} = \overline{ran}D_{X^*}$ \cite{NF}.

In the context of quotient modules, the above theorem can be stated as follows (see Theorem \ref{pure_isometry}): Let $\clq \subseteq H^2_{\cle_*}(\D^n)$ be a quotient module. Then $\clq$ is a Beurling quotient module if and only if
\[
z_i\cld_{C_j}\subseteq \cld_{C_j},
\]
for all $i\neq j$. In particular, the restriction operator $M_{z_i}|_{\cld_{C_j}}$ represents a shift on $\cld_{C_j}$ for all $i \neq j$.

These results provide a conclusive answer to the problem posed in part 1 of Question \ref{quest}. In view of this, we can now identify the class on which to focus when computing the characteristic functions—namely, the Beurling tuples. More formally:

\begin{defn}\label{def: Beurl tuples}
An $n$-tuple $T \in \mathbb{S}_n(\clh)$ is said to be a Beurling tuple if $D_{T_i}D_{T_j}=0$ for all $i \neq j$. We write $\mathbb{S}_n^{B}(\clh)$ to denote the set of such $n$-tuples (the superscript $B$ refers to Beurling). Therefore:
\[
\mathbb{S}_n^{B}(\clh) = \{T\in  \mathbb{S}_n(\clh): D_{T_i}D_{T_j}=0 \text{ for all } i\neq j    \}.
\]
\end{defn}

Having established that the tuples in $\mathbb{S}_n^{B}(\clh)$ are the ones to look for characteristic functions, we proceed by highlighting the representations of the defect operators of these tuples. There are two kinds of defect operators to deal with. Defect operators of the first kind are associated with the Szeg\"{o} property of Szeg\"{o} tuples: Given $T$ from $\mathbb{S}_n(\clh)$, we define the defect operator of the first kind as
\[
D_{T^*} = \Big(\mathbb{S}_n^{-1}(T, T^*)\Big)^{\frac{1}{2}}.
\]
The defect space of the first kind of $T \in \mathbb{S}_n(\clh)$ is defined as
\[
\cld_{T^*} = \overline{\text{ran}}D_{T^*}.
\]
The construction of this defect operator and space is widely recognized and inherently linked to the canonical dilation of $T$ \cite{MV} (also see Theorem \ref{Dilation Thm}). The next step is to build the second defect operator and the defect space that goes along with it. Extensive work is required to build the defect operator of the second kind. Notably, during the construction of the defect operator of the second kind, new concepts and ideas are introduced. These are of independent interest and shed light on many classical problems related to lifting in the context of several variables. We will elaborate on this topic in subsequent sections; however, we presently emphasize only the key aspects of the construction of the other defect operator.

First, we pull out a standard notation. For each $X\in \clb(\clh)$, we define a completely positive map $\Delta_X:\clb{(\clh)}\raro \clb(\clh)$ by
\[
\Delta_X(A) = XAX^*,
\]
for all $A\in \clb(\clh)$. To construct the defect operator of the second kind, we need to introduce two new notions. The first notion relates to the commutators of commuting tuples of operators. Recall that the commutator of $X$ and $Y$ from $\clb(\clh)$ is defined by $[X, Y] = XY - YX$. Given $T \in \clb^n_c(\clh)$ and $i\neq j$ from $I_n$, $n > 2$, the \textit{$(i,j)$-th joint commutator} of $T$ is defined by (see Definition \ref{def: joint comm})
\[
\delta_{ij}(T) = \prod_{k \in I_n \setminus \{i,j\}}^{n}(I-\Delta_{T_k})([T_j,T_i^*]).
\]
If $n=2$, then we set $\delta_{12}(T) := [T_2, T_1^*]$. In subsequent sections, we will develop many features of joint commutators. The second component needed in the construction of the defect operator of the second kind is truncated defect operators. For a detailed elaboration, we refer the reader to Section \ref{sec: trunc defect}, but for the present purpose, for a given commuting tuple of contractions $T \in \clb^n_c(\clh)$ and for $j \in I_n$ the \textit{truncated defect operator} is defined as (see \eqref{eqn: D jT})
\[
D^2_{j,T} = \prod_{k \in I_n \setminus \{j\}} (I - \Delta_{T_k})(D^2_{T_j}).
\]
An important point is that if $T$ is a Beurling tuple (that is, $T \in \mathbb{S}_n^{B}(\clh)$), truncated defect operators are positive. Furthermore, in this case, we define an operator $\bm{D}_T^2 \in \clb(\clh^n)$ by (see Definition \ref{defect})
\[
\bm{D}_T^2:=\begin{bmatrix}
D_{1, T}^2 & \delta_{12}(T) &\cdots & \delta_{1n}(T)
\\
\delta_{21}(T) & D_{2, T}^2&\cdots & \delta_{2n}(T)
\\
\vdots&\vdots&\ddots&\vdots\\
\delta_{n1}(T)  & \delta_{n2}(T) &\cdots&D_{n, T}^2
\end{bmatrix},
\]
where $\clh^n$ represents the direct sum of $n$ copies of $\clh$. Although $\bm{D}_T^2$ makes sense for Szeg\"{o} tuples, for Beurling tuples, this carries the much-needed property for our context. Indeed, in Proposition \ref{positive_defect}, we prove, for each $T \in \mathbb{S}_n^{B}(\clh)$, that
\[
\bm{D}_T^2 \geq 0.
\]
In other words, the superscript $2$ in $\bm{D}_T^2$ for Beurling tuples $T$ makes perfect sense. In view of this, for $T \in \mathbb{S}_n^{B}(\clh)$, we write the non-negative square root of $\bm{D}_T^2$ on $\clh^n$ as
\[
\bm{D}_T := (\bm{D}_T^2)^{\frac{1}{2}}.
\]
This is our defect operator of the second kind, which will be referred to as the \textit{joint defect operator} of $T$. The \textit{joint defect space} $\bm{\cld}_T$ of $T$ is then defined by
\[
\bm{\cld}_T := \overline{\text{ran}}\bm{D}_T \subseteq \clh^n.
\]
We remark a very curious fact that $\bm{\cld}_T$ can be represented as a closed subspace of $\clh$ due to certain orthogonality properties such as $D_{i,T} D_{j,T} = 0$ for all $i \neq j$. In fact, in \eqref{eqn: D D H}, we prove that
\[
\bm{\cld}_T \subseteq \bigoplus_{i \in I_n} \cld_{i,T} \subseteq \clh,
\]
where $\cld_{j,T} = \overline{\text{ran}}D_{j,T}$ for all $j \in I_n$. Another fascinating feature that significantly contributes to the formulation of characteristic functions is the connection between joint defect operators and wandering subspaces: Given a Beurling quotient module $\clq_{\Theta}:=H^2_{\cle_*}(\D^n)/ \Theta H^2_{\cle}(\D^n)$ for some inner function $\Theta\in H^{\infty}_{\clb(\cle,\cle_*)}(\D^n)$, the tuple of model operators $C = (C_1, \ldots,C_n)$ is a Beurling tuple. The wandering subspace of the \textit{Beurling submodule} $\cls_\Theta =  \Theta H^2_{\cle}(\D^n)$ is defined by
\[
\clw = \bigcap_{i\in I_n} (\cls_\Theta \ominus z_i \cls_\Theta).
\]
In Corollary \ref{unitary}, we prove that
\[
\text{dim} \clw = \dim \bm{\cld}_C.
\]
In fact, the unitary that implements the above identity is explicit and plays an important role in the structure of tuples of commuting contractions.

Finally, we are in a position to introduce characteristic functions of Beurling tuples (see Definition \ref{def: char funct}). We reiterate to the reader that the assumption that the tuples in the following are of Beurling type is not merely artificial; it is essential for a tuple to admit a characteristic function.

\begin{defn}
The characteristic function of a Beurling tuple $T \in \mathbb{S}_n^B(\clh)$ is as an operator-valued analytic function $\Theta_T : \D^n \raro \clb(\bm{\cld}_T, \cld_{T^*})$ defined by
\[
\Theta_{T}(w)\bm{D}_T \tilde{h} = D_{T^*}\prod_{k=1}^{n}(I_{\clh}-w_kT_k^*)^{-1}\sum_{j=1}^{n}(w_jI_{\clh}-T_j)\prod_{i \in I_n \setminus \{j\}} (I_{\clh}-w_iT_i^*)h_j,
\]
for all $w \in \D^n$ and $\tilde{h} = \begin{bmatrix}h_1 & \ldots & h_n \end{bmatrix}^t \in \clh^n$.
\end{defn}

This function is canonical and concrete, and subsequently, in Theorem \ref{thm: compl unit inv}, we prove that this is a complete unitary invariant. This solves the problem stated in part 2 of Question \ref{quest}. This also yields a concrete model representation of Beurling tuples. More specifically, given $T \in \mathbb{S}_n^B(\clh)$ and the characteristic function $\Theta_T$ of $T$ as described above, we can conclude that $T$ is jointly unitarily equivalent with the tuple $C = (C_1, \ldots, C_n)$ on the Beurling quotient module
\[
\clq_{\Theta_T} = H^2_{\cld_{T^*}}(\D^n)/ \Theta_T H^2_{\bm{\cld}_T}(\D^n).
\]

Finally, in Subsection \ref{sub sec: inner funct}, we use this tool to describe inner functions on $\D^n$. More specifically, if $\Theta \in H^\infty_{\clb(\cle,\cle_*)}(\D^n)$ is an inner function, then there exists a Beurling tuple $T \in \mathbb{S}_n^B(\clh)$, a Hilbert space $\clf$, and unitary operators $\sigma_*$ and $\sigma$ acting on some appropriate spaces so that (see Subsection \ref{sub sec: inner funct} for detailed constructions)
\[
\Theta(z) = \sigma_*^{*}\begin{bmatrix}
\Theta_T(z)&0\\0&I_{\clf}
\end{bmatrix}\sigma,
\]
for all $z \in \D^n$. This solves the remaining problem stated in part 3 of Question \ref{quest}, thereby resolving the full collection of questions about the characteristic functions of $n$-tuples of pure and commuting contractions and representations of inner functions on $\D^n$.

During the construction of characteristic functions of Beurling tuples, we develop numerous concepts, such as joint commutators, which are of independent interest and suggest new avenues for investigating many classical theories in several variables. In addition, we prove some results that are not directly related to the main goal of this paper but are still of independent interest, including the result that Beurling quotients are always infinite dimensional, thereby recovering and generalizing a classical result by Ahern and Clark (see Subsection \ref{subsect: Ahern and Clark}).

It is important to that the notion of characteristic functions for noncommuting tuples was initiated by Popescu in \cite{Pop}. However, our commutative approach is fundamentally different from Popescu’s methodology. His constructions and representations of characteristic functions for noncommuting tuples on noncommuting polyballs heavily rely on the presence of noncommutative Berezin transforms (maps that are essentially related to the canonical dilations of noncommutative tuples). See the definition and discussion preceding Theorem 6.3 of \cite{Pop}.

The paper is organized as follows. Section \ref{sec: basics} introduces all the necessary definitions and notations and recalls the classical characteristic functions on the disc. Section \ref{sec: Beurling QM} analyzes several properties of Szeg\"{o} tuples and also presents yet another classification of Beurling quotient modules over $\D^n$, $n > 1$. In Section \ref{sec: trunc defect}, we define truncated defect operators as a preliminary step toward introducing defect operators of the second kind, later referred to as joint defect operators. Section \ref{sec: trunc defect} also yields power series representations of truncated defect operators. Section \ref{sec: joint comm} introduces a new notion referred to as joint commutators. This section again examines the Beurling quotient modules. Building on concepts from the previous sections, this section demonstrates several structural properties of these quotient modules. The formal definition and some properties of joint defect operators appear in Section \ref{sec: defect op}. With these tools in hand, in Section \ref{sec: ch fn}, we introduce the characteristic functions of Beurling tuples and prove that they are complete unitary invariants. The final section, Section \ref{sec: classical}, presents that our characteristic function recovers the classical one when $n=1$. Here, we also prove that the representations of characteristic functions indeed lead to representations of inner functions. Additionally, we revisit and generalize a classical result by Ahern and Clark.

\section{Basic settings}\label{sec: basics}

The purpose of this section is to introduce basic notation and concepts and to collect key results from the literature. We begin by outlining the vector-valued Hardy space on $\D^n$. Once again, we recall that we primarily assume $n \geq 1$, and it will be clear from the context when $n=1$ is not included.

Let $\cle$ be a Hilbert space. The $\cle$-valued Hardy space over $\D^n$, denoted by $H^2_{\cle}(\D^n)$, is the Hilbert space of all $\cle$-valued analytic functions $f$ on $\D^n$ such that
\[
\|f\|:=\left(\sup_{0<r< 1}\int_{\mathbb{T}^n}\|f(rz)\|_\cle^2\: dm(z)\right)^{1/2}<\infty,
\]
where $dm$ denotes the normalized Lebesgue measure on $\mathbb{T}^n$ and $rz=(rz_1,\dots,rz_n)$, and $\mathbb{T}^n$ the distinguished boundary of $\D^n$. In the case of $\cle = \mathbb{C}$, we write $H^2(\D^n)$ in place of $H^2_{\mathbb{C}}(\D^n)$. It will be often useful to identify $H^2(\D^n)\otimes \cle$ with $H^2_{\cle}(\D^n)$ via the unitary map $U: H^2(\D^n)\otimes \cle \raro H^2_{\cle}(\D^n)$ defined by
\[
U (z^k \otimes \eta) = z^k \eta,
\]
for all $k \in \Z_+^n$ and $\eta \in \cle$, where $z^k = z_1^{k_1} \cdots z_n^{k_n}$. Recall that the $\clb(\cle)$-valued function
\[
\D^n\times \D^n \ni  (z, w) \mapsto  \mathbb{S}_n(z, w)I_{\cle},
\]
defines a reproducing kernel for $H^2_{\cle}(\D^n)$, where $\mathbb{S}_n$ is the Szeg\"{o} kernel of $\D^n$ given by
\[
\bS_n(z, w) = \prod_{i \in I_n} (1 - z_i \bar{w}_i)^{-1},
\]
for all $z, w \in \D^n$. Now we turn to operators. For each $i \in I_n$, define the multiplication operator $M_{z_i}$ on $H^2_{\cle}(\D^n)$ by
\[
(M_{z_i}f)(w): = w_i f(w),
\]
for all $f\in H^2_{\cle}(\D^n)$ and $w \in \D^n$. Then, $(M_{z_1},\dots, M_{z_n})$ is an $n$-tuple of doubly commuting isometries (this is also referred to as \textit{shift operators}). A closed subspace $\clq$ of $H^2_{\cle}(\D^n)$ is referred to as a \textit{backward shift-invariant subspace} or a \textit{quotient module} if
\[
M_{z_i}^*\clq\subseteq \clq,
\]
for all $i \in I_n$. We further say that $\clq^\perp$ is a \textit{submodule} of $H^2_{\cle}(\D^n)$. Note that a closed subspace $\cls$ of $H^2_{\cle}(\D^n)$ is a submodule of $H^2_{\cle}(\D^n)$ if and only if $z_i \cls \subseteq \cls$ for all $i \in I_n$.

This $n$-tuple $(M_{z_1},\dots, M_{z_n})$ on $H^2_{\cle}(\D^n)$ is also an example of Szeg\"{o} tuple. To recall the definition of Szeg\"{o} tuples, we bring back the Szeg\"{o} kernel of $\D^n$. We observe that $\mathbb{S}^{-1}_n(z, w)$ is a polynomial:
\[
\mathbb{S}^{-1}_n(z, w) = \sum_{k \in \{0,1\}^n} (-1)^{|k|} z^k \bar{w}^{k},
\]
where $|k| = \sum_{i=1}^{n} k_i$ and $\bar{w}^k = \bar{w}_1^{k_1} \cdots \bar{w}_n^{k_n}$ for all $k\in \Z_+^n$. In view of the polynomial $\mathbb{S}^{-1}_n(z, w)$, for each $T \in \clb_c^n(\clh)$, we set (also see the identity preceding Definition \ref{def: Szego tuples})
\begin{equation}\label{eqn: Szego inverse}
\mathbb{S}^{-1}_n(T, T^*) := \sum_{k \in \{0,1\}^n}(-1)^{|k|}T^k T^{*k}.
\end{equation}

As mentioned in the introduction, in the classical setting of the $n=1$ case, the concept of characteristic functions connects the structure of operators and analytic function theory. The existence of isometric dilations enriches the link. However, due to the complex structure of tuples of commuting isometries, the existence of isometric dilations in the polydisc setting remains a challenging problem (however, see \cite{Isometric dilation, Vinnikov}). On the other hand, tuples that admit dilations to tuples of shifts are well understood and widely recognized. Before we recall the dilation result, we set up a few notations: Let $T\in \mathbb{S}_n(\clh)$. The \textit{defect operator of the first kind}, or simply the \textit{defect operator} $D_{T^*}$ of $T$, is defined by
\[
{D}_{T^*} = (\mathbb{S}^{-1}_n(T,T^*))^{\frac{1}{2}},
\]
and, the \textit{defect space} is defined by
\[
{\cld}_{T^*} = \overline{ran}{{D}_{T^*}}.
\]
Define a bounded linear operator $\Pi_T: \clh \raro H^2_{{\cld}_{T^*}}(\D^n)$ by
\[
(\Pi_Th)(w)={D}_{T^*} \prod_{i=1}^{n}(I_{\clh}-w_iT^*)^{-1}h,
\]
for all $w\in\D^n$ and $h\in \clh$. The following isometric dilation result for Szeg\"{o} tuples aligns with the classical Sz.-Nagy and Foias isometric dilation theorem. This result is due to M\"{u}ller and Vasilescu \cite{MV} (also see \cite{Timotin}).

\begin{thm}\label{Dilation Thm}
Let $T\in \mathbb{S}_n(\clh)$. Then $\Pi_T: \clh \raro H^2_{{\cld}_{T^*}}(\D^n)$ defines an isometry and satisfies
\[
\Pi_T T_i^* = M_{z_i}^*\Pi_T,
\]
for $i \in I_n$. Moreover, we have the following identity:
\[
H^2_{{\cld}_{T^*}}(\D^n)=\bigvee_{{k}\in \mathbb{Z}^n_{+}} {z}^{{k}} \Pi_T \clh.
\]
\end{thm}

The later cyclicity-type property of $T$, known as the \textit{minimality property}, will be crucial for our study. We will refer to the map $\Pi_T$ as the \textit{canonical dilation} of the Szeg\"{o} tuple $T$. By the intertwining property of $\Pi_T$, it follows that
\[
\clq_T := \Pi_T \clh,
\]
is a quotient module of $H^2_{\cld_{T^*}}(\D^n)$, and $T$ is unitary equivalent to $C_{\clq_T} :=(C_1,\dots, C_n)$ on $\clq_T$, which we write as
\begin{equation}\label{eqn: T and CT}
T \cong C_{\clq_T},
\end{equation}
where $C_i$, defined by
\[
C_i = P_{\clq_T} M_{z_i}|_{\clq_T},
\]
for all $i \in I_n$, is the \textit{model operator}. In general, given a quotient module $\clq \subseteq H^2_{\cle}(\D^n)$, for each $i \in I_n$, define the model operator $C_i$ as the compression of the shift $M_{z_i}$ on $\clq$, that is,
\[
C_i=P_{\clq}M_{z_i}|_{\clq}.
\]
Often we write $C_\clq$ (or simply $C$ if $\clq$ is clear from the context) to denote the commuting $n$-tuple $(C_1,\dots,C_n)$ on $\clq$. Note that $C_\clq \in \mathbb{S}_n(\clh)$.

In particular, representations of commuting tuples of contractions rely heavily on the representations of quotient modules of the Hardy spaces over $\D^n$. When $n$ equals $1$, the representations of quotient modules of vector-valued Hardy spaces become more lucid. To talk about this, we need to introduce inner functions. Let $\cle$ and $\cle_*$ be Hilbert spaces. Recall that $H^\infty_{\clb(\cle,\cle_*)}(\D^n)$ denotes the space of all bounded $\clb(\cle,\cle_*)$-valued analytic functions on $\D^n$. Recall also that a function $\Theta\in H^{\infty}_{\clb(\cle,\cle_*)}(\D^n)$ is \textit{inner} if
\[
\Theta(z)^*\Theta(z)=I_{\cle},
\]
for $z\in \mathbb{T}^n$ a.e. Given an inner function $\Theta\in H^{\infty}_{\clb(\cle,\cle_*)}(\D^n)$, we set the model space $\clq_\Theta$ as
\[
\clq_{\Theta}:=H^2_{\cle_*}(\D^n)/ \Theta H^2_{\cle}(\D^n).
\]
A simple computation reveals that $\clq_{\Theta}$ is a quotient module of $H^2_{\cle_*}(\D^n)$. With this note in mind, we define:

\begin{defn}\label{def: beurling QM}
A quotient module $\clq$ of $H^2_{\cle_*}(\D^n)$ is said to be a Beurling quotient module if there exist a Hilbert space $\cle$ and an inner function $\Theta\in H^{\infty}_{\clb(\cle,\cle_*)}(\D^n)$ such that
\[
\clq = \clq_\Theta.
\]
The submodule $\cls_\Theta$, where
\[
\cls_\Theta = \Theta H^2_{\cle}(\D^n),
\]
will be referred to as a Beurling submodule of $H^2_{\cle_*}(\D^n)$.
\end{defn}

In the case of $n=1$, due to Beurling, Lax, and Halmos, all quotient modules of $H^2_{\cle_*}(\D)$ qualify as Beurling quotient modules \cite[Chap. V, Theorem 3.3]{NF}. However, numerous quotient modules of $H^2_{\cle_*}(\D^n)$, $n > 1$, do not conform to Beurling type (see Rudin \cite{Rudin}). This is in fact one of the aspects of the complexity of multivariable operator theory and function theory on $\D^n$.

We conclude this section by discussing a property of classical characteristic functions that provides some (but not all) insight into the construction process of characteristic functions on $\D^n$, $n > 1$. The Sz.-Nagy and Foias characteristic function of a $C_{\cdot 0}$ contraction $T$ is the inner function $\Theta_T:\D\raro\clb(\cld_T,\cld_{T^*})$ defined by (see \eqref{eqn: defect} for relevant notations)
\begin{equation}\label{eqn: ch fn 1 var}
\Theta_T(w)=[-T+wD_{T^*}(I_{\clh}-wT^*)^{-1}D_T]|_{\cld_T},
\end{equation}
for all $w\in \D$. In view of $TD_T=D_{T^*}T$, it follows that $\Theta_T$ is a $\clb(\cld_T, \cld_{T^*})$-valued analytic function on $\D$. Moreover, for each $h\in \clh$ and $w\in \D$, we compute
\[
\begin{split}
\Theta_T(w)D_Th&=-TD_Th+wD_{T^*}(I_{\clh}-wT^*)^{-1}D_T^2h
\\
&=-D_{T^*}Th+wD_{T^*}(I_{\clh}-wT^*)^{-1}D_T^2h
\\
&=-D_{T^*}Th+wD_{T^*}(I_{\clh}-wT^*)^{-1}h-wD_{T^*}(I_{\clh}-wT^*)^{-1}T^*Th
\\
&=wD_{T^*}(I_{\clh}-wT^*)^{-1}h -D_{T^*}(I+w(I_{\clh}-wT^*)^{-1}T^*)Th
\\
&=wD_{T^*}(I_{\clh}-wT^*)^{-1}h-D_{T^*}(I_{\clh}-wT^*)^{-1}Th,
\end{split}
\]
which implies
\[
\Theta_T(w)D_T = D_{T^*}(I_{\clh}-wT^*)^{-1}(wI_{\clh}-T).
\]
This is an alternative expression of the Sz.-Nagy and Foias characteristic function (see \cite[the identity (1.2) on page 237]{NF}) and also relevant for our consideration. Indeed, when $n=1$, our approach to constructing the characteristic function will align with the above description (compare the above with Definition \ref{def: char funct}).

\section{Quotient modules}\label{sec: Beurling QM}

This section examines Beurling quotient modules of vector-valued Hardy spaces over $\D^n$, $n > 1$. Concrete classifications of Beurling quotient modules were obtained in \cite{Beurling quotient module}. Here we present yet another characterization. Along the way, we continue to develop tools for the main goal of this paper, each of which is of independent interest.

 
\begin{lem}\label{lemma: Q redu}
Let $\clq \subseteq H^2_{\cle_*}(\D^n)$ be a quotient module. Then $\clq$ reduces the $n$-tuple
\[
(M_{z_1}^*P_{\clq^{\perp}}M_{z_1}, \ldots, M_{z_n}^*P_{\clq^{\perp}}M_{z_n}).
\]
\end{lem}
\begin{proof}
We need to prove that $P_{\clq}(M_{z_i}^*P_{\clq^{\perp}}M_{z_i})=( M_{z_i}^*P_{\clq^{\perp}}M_{z_i})P_{\clq}$ for all $i \in I_n$. To this end, fix $i \in I_n$. By the invariance property, we know that $P_{\clq^\perp} M_{z_i}^*P_{\clq^{\perp}} = P_{\clq^\perp} M_{z_i}^*$, and hence $P_{\clq^\perp} M_{z_i}^*P_{\clq^{\perp}}M_{z_i} = P_{\clq^\perp}$. We now compute
\[
\begin{split}
P_{\clq}(M_{z_i}^*P_{\clq^{\perp}}M_{z_i}) & = (I - P_{\clq^\perp}) (M_{z_i}^*P_{\clq^{\perp}}M_{z_i})
\\
& = M_{z_i}^*P_{\clq^{\perp}}M_{z_i} - P_{\clq^{\perp}}
\\
& = M_{z_i}^*P_{\clq^{\perp}}M_{z_i} (P_{\clq} + P_{\clq^{\perp}}) - P_{\clq^{\perp}}
\\
& = M_{z_i}^*P_{\clq^{\perp}}M_{z_i} P_{\clq} + P_{\clq^{\perp}} - P_{\clq^{\perp}}
\\
& = (M_{z_i}^*P_{\clq^{\perp}}M_{z_i}) P_{\clq}.
\end{split}
\]
This completes the proof of the lemma.
\end{proof}

Now we turn to representations of defect operators and defect spaces of model operators. It is useful to review the definition of defect operators (see \eqref{eqn: defect}).

\begin{lem}\label{lemma: defect Ci}
Let $\clq$ be a quotient module of $H^2_{\cle_*}(\D^n)$, and let $i \in I_n$. Then
\[
D_{C_i}^2 = P_{\clq}M_{z_i}^*P_{\clq^{\perp}}M_{z_i}|_{\clq},
\]
and
\[
\cld_{C_i} = \overline{\text{ran}}(P_{\clq}M_{z_i}^*P_{\clq^{\perp}}).
\]
\end{lem}
\begin{proof}
Since $D_{C_i}^2 = I_{\clq}-C_i^*C_i$, and $C_i = P_\clq M_{z_i}|_\clq$, we have
\[
D_{C_i}^2 = I_{\clq}-P_{\clq}M_{z_i}^*P_{\clq}M_{z_i}|_{\clq} = I_{\clq}-P_{\clq}M_{z_i}^*(I-P_{\clq^{\perp}})M_{z_i}|_{\clq} = P_{\clq}M_{z_i}^*P_{\clq^{\perp}}M_{z_i}|_{\clq}.
\]
This also shows that $\cld_{C_i}= \overline{\text{ran}} D_{C_i}^2 = \overline{\text{ran}}(P_{\clq}M_{z_i}^*P_{\clq^{\perp}})$.
\end{proof}

We now turn to Szeg\"{o} tuples. Let $T\in \mathbb{S}_n(\clh)$. According to Theorem \ref{Dilation Thm} and \eqref{eqn: T and CT}, $T \cong C$. Consequently, the study of Szeg\"{o} tuples is synonymous with the analysis of analogous tuples of model operators. We will frequently utilize this equivalency interchangeably, occasionally even without explicitly referencing the link. 

We wish to present approximation results in terms of operators and subspaces of Szeg\"{o} tuples. The following is the first step to that end. Among many other things, this result also says a lot about invariant subspaces of $T_i$, $i \in I_n$.
 

\begin{thm}\label{More dilation}
$T\in \mathbb{S}_n(\clh)$, and let $i, j \in I_n$. Assume that $i \neq j$. Then we have the following:
\begin{enumerate}
\item[(1)] $T_j\cld_{T_i}\subseteq \cld_{T_i}$.
\item[(2)] $T_jD^2_{T_i}T_j^*\leq  D^2_{T_i}$.
\item[(3)] $T_jD_{T_i^*}^2T_j^*\leq D_{T_i^*}^2$.
\item[(4)] $T_j\cld_{T_i^*}\subseteq \cld_{T_i^*}$.
\end{enumerate}
\end{thm}
\begin{proof}
In view of Theorem \ref{Dilation Thm} (or \eqref{eqn: T and CT}, to be more specific), there exists a quotient module $\clq$ of $H^2_{\cle}(\D^n)$, for some Hilbert space $\cle$, such that $T\cong C$, where
\[
C = (P_\clq M_{z_1}|_\clq, \ldots, P_\clq M_{z_n}|_\clq).
\]
So we prove the result for $C$ instead of $T$. Lemma \ref{lemma: defect Ci} implies $D_{C_t}^2=P_{\clq}M_{z_t}^*P_{\cls}M_{z_t}|_{\clq}$ for all $t \in I_n$, where
\[
\cls = \clq^\perp.
\]
Now, we observe that $M_{z_t} P_\cls M_{z_t}^*$ is an orthogonal projection, and moreover, we have
\begin{equation}\label{eqn: proj onto zS}
M_{z_t} P_\cls M_{z_t}^* = P_{z_t \cls},
\end{equation}
for all $t \in I_n$. Suppose $i\neq j$. Since $C_j = P_{\clq}M_{z_j}|_{\clq}$, it follows that
\[
C_jD_{C_i}^2 = P_{\clq}M_{z_j}P_{\clq} M_{z_i}^*P_{\cls}M_{z_i}|_{\clq}.
\]
The reducing property of Lemma \ref{lemma: Q redu} implies that $P_\clq M_{z_i}^*P_{\cls}M_{z_i} = M_{z_i}^*P_{\cls}M_{z_i} P_\clq $, and hence
\[
C_jD_{C_i}^2 = P_{\clq}M_{z_j}M_{z_i}^*P_{\cls}M_{z_i}|_{\clq} = P_{\clq} M_{z_i}^* M_{z_j} P_{\cls}M_{z_i}|_{\clq} = (P_{\clq}M_{z_i}^* P_{\cls})M_{z_j}P_{\cls} M_{z_i}|_{\clq}.
\]
In the above, we have used the facts that $M_{z_i}^* M_{z_j} = M_{z_j} M_{z_i}^*$ and $ P_{\cls} M_{z_j}P_{\cls} = M_{z_j} P_{\cls}$ (recall that $\cls$ is invariant under $M_{z_i}$). By Lemma \ref{lemma: defect Ci}, we know that $\cld_{C_i}=\overline{ran}(P_{\clq}M_{z_i}^*|_{\cls})$, and consequently (by the Douglas' range inclusion theorem)  $C_j\cld_{C_i}\subseteq \cld_{C_i}$. This proves (1). By Lemma \ref{lemma: defect Ci} again, we have
\[
C_jD_{C_i}^2C_j^* = P_{\clq}M_{z_i}^*M_{z_j}P_{\cls} M_{z_j}^* M_{z_i}|_{\clq} = P_{\clq}M_{z_i}^* (M_{z_j}P_{\cls} M_{z_j}^*) M_{z_i}|_{\clq}.
\]
As \eqref{eqn: proj onto zS} implies $M_{z_i} P_\cls M_{z_i}^* = P_{z_i \cls}$, it follows that
\begin{equation}\label{eqn: CjDiCj}
C_jD_{C_i}^2C_j^* = P_{\clq}M_{z_i}^*P_{z_j\cls} M_{z_i}|_{\clq}.
\end{equation}
Since $P_{z_j\cls}\leq P_{\cls}$, we conclude that $P_{\clq}M_{z_i}^*P_{z_j\cls} M_{z_i} P_{\clq} \leq P_{\clq}M_{z_i}^*P_{\cls} M_{z_i} P_{\clq}$, which implies
\[
C_jD_{C_i}^2C_j^*\leq P_{\clq}M_{z_i}^*P_{\cls}  M_{z_i}|_{\clq}=D_{C_i}^2,
\]
and completes the proof of part (2). For (3), we simply compute
\[
\begin{split}
D_{C_i^*}^2-C_j D_{C_i^*}^2 C_j^*&=I-C_iC_i^*-C_jC_j^*+C_iC_jC_i^*C_j^*
\\
&=P_{\clq}(I-M_{z_i} M_{z_i}^*-M_{z_j} M_{z_j}^*+M_{z_i}M_{z_j} M_{z_i}^*M_{z_j}^*)|_{\clq}
\\
&=P_{\clq}(I-M_{z_i} M_{z_i}^*)(I-M_{z_j} M_{z_j}^*)|_{\clq},
\end{split}
\]
as $M_{z_i}$ and $M_{z_j}$ are doubly commuting. The doubly commuting property also implies that $\{(I-M_{z_i} M_{z_i}^*), (I-M_{z_j} M_{z_j}^*)\}$ is a pair of commuting projections, and subsequently, the product $(I-M_{z_i} M_{z_i}^*)(I-M_{z_j} M_{z_j}^*)$ itself is a projection. The above, then, implies
\[
D_{C_i^*}^2-C_j D_{C_i^*}^2 C_j^* \geq 0,
\]
which proves that $C_j D_{C_i^*}^2 C_j^*\leq D_{C_i^*}^2$. In particular, this also implies $C_j\cld_{C_i^*}\subseteq \cld_{C_i^*}$ and completes the proof of (4).
\end{proof}

We previously pointed out in Theorem \ref{Dilation Thm} (which comes from \cite{MV}) that a Szeg\"{o} tuple $T \in \mathbb{S}_n(\clh)$ admits the following model:
\[
T \cong (P_{\clq} M_{z_1}|_{\clq}, \ldots, P_{\clq} M_{z_n}|_{\clq}),
\]
where $\clq$ is a quotient module of a vector-valued Hardy space $H^2_{{\cld}_{T^*}}(\D^n)$. A natural question arises: under what additional conditions on the tuple $T$ does the quotient module $\clq$ become a Beurling quotient module? In other words, we want
\[
\clq = \clq_{\Theta}:= H^2_{\cle_*}(\D^n)/\Theta H^2_{\cle}(\D^n),
\]
for some inner function $\Theta\in H^{\infty}_{\clb(\cle,\cle_*)}(\D^n)$ and Hilbert spaces $\cle$ and $\cle_*$. In \cite{Beurling quotient module}, a classification was provided, which is the first two equivalences stated in the following result. The additional equivalent condition, (3), is a new addition.

\begin{thm}\label{Beurling_qm}
Let $T \in  \mathbb{S}_n(\clh)$. The following are equivalent:
\begin{enumerate}
\item[(1)] There exist Hilbert spaces $\cle$ and $\cle_*$ and inner function $\Theta\in H^{\infty}_{\clb(\cle,\cle_*)}(\D^n)$ such that $T \cong (P_{\clq_{\Theta}}M_{z_1}|_{\clq_{\Theta}},\dots, P_{\clq_{\Theta}}M_{z_n}|_{\clq_{\Theta}})$.
\item[(2)] $(I_{\clh}-T_i^*T_i)(I_{\clh}-T_j^*T_j)=0$ for all $i\neq j$.
\item[(3)] $T_i|_{\cld_{T_j}}: \cld_{T_j} \raro \cld_{T_j}$ is an isometry for all $i\neq j$.
\end{enumerate}
\end{thm}
\begin{proof}
As pointed out, the equivalence of (1) and (2) has already been proven in \cite{Beurling quotient module}. Here we show that (2) is equivalent to (3).
Suppose that $(I-T_j^*T_j)(I-T_i^*T_i)=0$. Then
\[
(I-T_j^*T_j)(I-T_i^*T_i)^{\frac{1}{2}}=0.
\]
In particular, for $h\in \clh$, we have $(I - T_j^*T_j)D_{T_i}h=0$, and hence $T_j^*T_j D_{T_i}h = D_{T_i}h$. Then
\[
\|T_jD_{T_i}h\|^2 = \la T_j^*T_jD_{T_i}h, D_{T_i}h\ra=\la D_{T_i}h, D_{T_i}h \ra,
\]
and hence
\[
\|T_jD_{T_i}h\|=\|D_{T_i}h\|.
\]
Now, part (1) of Theorem \ref{More dilation} says that $T_j\cld_{T_i}\subseteq \cld_{T_i}$, and consequently,  $\|T_j|_{\cld_{T_i}}f\|=\|f\|$ for all $f\in \cld_{T_i}$. This proves that (2) implies (3). For the reverse direction, assume that $f=D_{T_i}h$ for some $h\in \clh$. Then $\|T_j|_{\cld_{T_i}}f\|=\|f\|$. This implies
\[
\la D_{T_i}T_j^* T_jD_{T_i}h,h\ra=\la D^2_{T_i}h,h \ra,
\]
and so, $\la D_{T_i}(I-T_j^*T_j)D_{T_i}h,h \ra=0$. In particular, $\|(I-T_j^*T_j)^{\frac{1}{2}}D_{T_i}h\|=0$, which shows that $D_{T_j}D_{T_i}=0$. Then, we have
\[
(I-T_j^*T_j)(I-T_i^*T_i)=D_{T_j}^2D_{T_i}^2=0,
\]
which proves that (3) implies (2), and completes the proof of the theorem.
\end{proof}

Now, we apply this framework to quotient modules. Specifically, we consider $T$ as $C$ on some quotient module of a vector-valued Hardy space. In this specific context, we obtain additional information about these quotient modules due to the concrete nature of the situation.

\begin{thm}\label{pure_isometry}
Let $\clq\subseteq H^2_{\cle_*}(\D^n)$ be a quotient module. Then the following are equivalent:
\begin{enumerate}
\item[(1)] $\clq$ is a Beurling quotient module.
\item[(2)] $(I_{\clq}-C_i^*C_i)(I_{\clq}-C_j^*C_j)=0$ for all $i\neq j$.
\item[(3)] $C_i|_{\cld_{C_j}}: \cld_{C_j} \raro \cld_{C_j}$ is an isometry for all $i\neq j$.
\item[(4)] $z_i\cld_{C_j}\subseteq \cld_{C_j}$ for all $i\neq j$.
\end{enumerate}
\end{thm}	
	
\begin{proof}
Theorem \ref{Beurling_qm} shows that (1), (2), and (3) are equivalent. Our strategy is to show that $(1)\Rightarrow (4) \Rightarrow (3)$. Let $i\neq j$. For $(1)\Rightarrow (4)$, assume that $\clq$ is a Beurling quotient module. Then $C := C_\clq \in \mathbb{S}_n(\clq)$. We observed in the proof of Theorem \ref{More dilation} that
\[
C_jD_{C_i}^2 = P_{\clq}M_{z_j}M_{z_i}^*P_{\cls}M_{z_i}|_{\clq}.
\]
We claim that the projection factor $P_{\clq}$ in this representation is irrelevant. Assume that $\clq=\clq_{\Theta}$ for some inner function $\Theta \in H^\infty_{\clb(\cle, \cle_*)}(\D^n)$ and Hilbert space $\cle_*$. In this case, we have $\cls = \clq_{\Theta}^{\perp}=\Theta H^2_{\cle}(\D^n)$. In the following computation, we use the doubly commuting property that $M_{z_i}M_{z_j}^* = M_{z_j}^*M_{z_i}$ and the fact that $M_{\Theta}^*M_{z_j}^* M_{z_i}M_\Theta = M_{z_j}^* M_{z_i}$:
\[
\begin{split}
(M_{z_i}^*P_{\cls}M_{z_i}M_{z_j}^*)\Theta H^2_{\cle}(\D^n)&= M_{z_i}^*M_{\Theta}M_{\Theta}^*M_{z_i}M_{z_j}^* \Theta H^2_{\cle}(\D^n)
\\
&
= M_{z_i}^*M_{\Theta} M_{\Theta}^*M_{z_j}^* M_{z_i}\Theta H^2_{\cle}(\D^n)
\\
&=  M_{z_i}^* M_{\Theta} M_{z_j}^* M_{z_i} H^2_{\cle}(\D^n)
\\
&=M_{\Theta}M_{z_j}^* H^2_{\cle}(\D^n)
\\
&\subseteq \Theta H^2_{\cle}(\D^n).
\end{split}
\]
From this, we deduce that $M_{z_j}M_{z_i}^*P_{\cls}M_{z_i} \clq \subseteq \clq$. We have also used the fact that $\clq = \clq_\Theta$, $(\Theta H^2_{\cle}(\D^n))^\perp = \clq_\Theta$, and $(M_{z_i}^*P_{\cls}M_{z_i}M_{z_j}^*)^* = M_{z_j}M_{z_i}^*P_{\cls}M_{z_i}$. With this in mind, we have
\[
C_jD_{C_i}^2 = P_{\clq}M_{z_j}M_{z_i}^*P_{\cls}M_{z_i}|_{\clq} = M_{z_j}M_{z_i}^*P_{\cls}M_{z_i}|_{\clq}.
\]
The reducibility property of Lemma \ref{lemma: Q redu} says $(M_{z_i}^*P_{\cls}M_{z_i}) P_{\clq} = P_\clq (M_{z_i}^*P_{\cls}M_{z_i})$, and then
\[
C_jD_{C_i}^2 = M_{z_j}P_{\clq}M_{z_i}^*P_{\cls}M_{z_i}|_{\clq}.
\]
Recalling from Lemma \ref{lemma: defect Ci} that $D_{C_i}^2 = P_{\clq}M_{z_i}^*P_{\clq^{\perp}}M_{z_i}|_{\clq}$, it follows that
\[
C_jD_{C_i}^2 = M_{z_j}D_{C_i}^2.
\]
By part (4) of Theorem \ref{More dilation}, $C_j\cld_{C_i}\subseteq \cld_{C_i}$, which ensures that $z_j\cld_{C_i}\subseteq \cld_{C_i}$. This proves that $(1)$ implies $(4)$. Finally, assume that (4) holds. Then $z_j\cld_{C_i}\subseteq \cld_{C_i}\subseteq \clq$. For all $h\in \clq$, we obtain
\[
M_{z_j}D_{C_i}h=P_{\clq}M_{z_j}D_{C_i}h=C_jD_{C_i}h,
\]
and we have
\[
\|C_jD_{C_i}h\|=\|M_{z_j}D_{C_i}h\|=\|D_{C_i}h\|.
\]
In other words, $C_j|_{\cld_{C_i}}$ is an isometry. This proves that (4) implies (3), and completes the proof of the theorem.
\end{proof}

The fourth condition in the above statement is particularly intriguing: it asserts that the restriction operator $M_{z_i}|_{\cld_{C_j}}$ represents a shift on $\cld_{C_j}$ for all $i \neq j$. This new observation will be very useful in what is to come.

For an $n$-tuple to admit a characteristic function that is inner in our context, it must satisfy one of the equivalence conditions in Theorem \ref{Beurling_qm}. We refer to such tuples as Beurling tuples. Recall from Definition \ref{def: Beurl tuples} that $T \in \mathbb{S}_n(\clh)$ is a Beurling tuple if and only if $D_{T_i}D_{T_j}=0$ for all $i \neq j$. Recall also that $\mathbb{S}_n^{B}(\clh)$ denotes the set of such Beurling $n$-tuples. In other words, $\mathbb{S}_n^{B}(\clh)$ represents the set of all tuples that require investigation to study characteristic functions. We now turn to analyze Beurling tuples.

\section{Truncated defect operators}\label{sec: trunc defect}

The properties detected in class $\mathbb{S}_n^{B}(\clh)$ thus far only identify those tuples for which characteristic functions exist. And we call them Beurling tuples. Naturally, our ultimate goal is to compute or explicitly represent these characteristic functions (of Beurling tuples). Significant work remains to achieve this goal, and we begin preparing the necessary tools to facilitate it. The first goal is to construct defect operators that are parameterized by subsets and elements of $I_n$. The computations involved are substantial and may contribute to future advancements in the field.

We begin by introducing some standard tools that are common in dilation theory \cite{MV}. Given $T\in \clb(\clh)$, define the completely positive map $\Delta_T:\clb{(\clh)}\raro \clb(\clh)$ by
\[
\Delta_T(X) =TXT^*,
\]
for all $X\in \clb(\clh)$. It is easy to see for commuting pair of operators $(T_1,T_2)$ acting on some Hilbert space that
\[
\Delta_{T_1}\Delta_{T_2}=\Delta_{T_2}\Delta_{T_1}.
\]
Recall the definition of $\mathbb{S}^{-1}_n(T, T^*)$ from \eqref{eqn: Szego inverse}. In view of the above, a simple computation, for each $T \in \clb(\clh)^n_c$, yields that
\[
\mathbb{S}^{-1}_n(T,T^*)= \prod_{i \in I_n} (I_{\clb(\clh)} - \Delta_{T_i})(I_{\clh}).
\]
Now, we are ready to introduce truncated defect operators, which will be placed on the diagonals of the defect operators of the second kind.
	
\begin{defn}\label{defn: trunc defect}
Let $T \in \clb^n_c(\clh)$. For each $j\in I_n$ and nonempty $\clp\subseteq I_n \setminus\{j\}$, we define the $(j, \clp)$-th truncated defect operator $D^2_{j,T,\clp}$ by
\[
D^2_{j,T,\clp}=\prod_{k\in \clp} (I_{\clb(\clh)}-\Delta_{T_k})(D_{T_j}^2).
\]
Moreover, we define $D^2_{j,T, \emptyset}:=D^2_{T_j}$.
\end{defn}

We will often suppress the pair $(j, \clp)$ from the definition as it will be clear from the context. In the case of $\clp= I_n \setminus \{j\}$, we simply write the defect $D^2_{j,T,\clp}$ as
\[
D^2_{j,T, I_n \setminus \{j\}} = D^2_{j,T}.
\]
In other words, for all $j \in I_n$, we set
\begin{equation}\label{eqn: D jT}
D^2_{j,T} = \prod_{k \in I_n \setminus \{j\}} (I_{\clb(\clh)} - \Delta_{T_k})(D^2_{T_j}).
\end{equation}

We remark that the truncated defect operators defined above need not be positive. In other words, as of now, the power $2$ in the definition of the defect operator is merely a matter of notation. However, we must focus on Beurling tuples. We need a notation: Let $\cls$ be a submodule of $H^2_{\cle_*}(\D^n)$. For each $\clp\subseteq I_n$, we define the \textit{$\clp$-th wandering subspace} $\clw_\clp$ by
\begin{equation}\label{defn: W}
\clw_{\clp}:=\bigcap_{i\in \clp}(\cls \ominus z_i\cls).
\end{equation}
We often refer to $\clw_{\clp}$ as a \textit{wandering subspace}, as the context will make $\clp$ evident.

The following lemma concerning doubly commuting submodules is well known and can be attributed partly to S{\l}oci\'{n}ski \cite{Slo} and Mandrekar \cite{Mand} (also see \cite{JayII}). We present proof of part of the statement.

\begin{lem}\label{wandering}
Let $\Theta\in H^{\infty}_{\clb(\cle,\cle_*)}(\D^n)$ be an inner function, and let $\cls_{\Theta}=\clq_{\Theta}^{\perp}$. Then, the following are true:

\begin{enumerate}
\item[(1)] $\clw=\Theta \cle$.
\item[(2)] $z_j(\cls_{\Theta}\ominus z_i\cls_{\Theta})\subseteq (\cls_{\Theta}\ominus z_i\cls_{\Theta})$ for all $i\neq j$.
\item[(3)] $z_j\clw_{\clp}\subseteq \clw_{\clp}$ for all $j \in I_n$ and $\clp\subseteq I_n \setminus \{j\}$.
\item[(4)]  $\clw_{\clp}\ominus z_j \clw_{\clp}=\clw_{\clp\cup \{j \}}$ for all $\clp\subsetneq I_n$, $\clp \neq \emptyset$, and $j\in I_n \setminus\clp$.
\end{enumerate}
\end{lem}
	
\begin{proof}
First, we observe that
\[
\clw = \bigcap_{i=1}^n\left(\cls_{\Theta}\ominus z_i\cls_{\Theta}\right) = \bigcap_{i=1}^n\left(\Theta H^2_{\cle}(\D^n)\ominus z_i\Theta H^2_{\cle}(\D^n) \right) = \Theta (\bigcap_{i=1}^n \left(H^2_{\cle}(\D^n)\ominus z_i H^2_{\cle}(\D^n)\right)).
\]
Since $\Theta (\bigcap_{i=1}^n \left(H^2_{\cle}(\D^n)\ominus z_i H^2_{\cle}(\D^n)\right)) = \ker M_{z_i}^*$, it follows that $\clw =\Theta \cle$. This proves (1). For (2), we compute
\[
z_j(\cls_{\Theta}\ominus z_i\cls_{\Theta}) = z_j\Theta \left(  H^2_{\cle}(\D^n)\ominus z_i H^2_{\cle}(\D^n)\right) = \Theta z_j\left(  H^2_{\cle}(\D^n)\ominus z_i H^2_{\cle}(\D^n)\right).
\]
Since $\Theta z_j\left(  H^2_{\cle}(\D^n)\ominus z_i H^2_{\cle}(\D^n)\right) \subseteq \Theta \left(H^2_{\cle}(\D^n)\ominus z_i H^2_{\cle}(\D^n)\right)$, it follows that $z_j(\cls_{\Theta}\ominus z_i\cls_{\Theta}) \subseteq (\cls_{\Theta}\ominus z_i\cls_{\Theta})$. For (3). we observe $\clw_{\clp}=\cap_{k\in\clp}(\cls_{\Theta}\ominus z_k\cls_{\Theta})$, and hence
\[
z_j\clw_{\clp}= z_j \bigcap_{k\in \clp}(\cls_{\Theta}\ominus z_k\cls_{\Theta})
=\bigcap_{k\in \clp} z_j(\cls_{\Theta}\ominus z_k\cls_{\Theta}) \subseteq \bigcap_{k\in \clp} (\cls_{\Theta}\ominus z_k\cls_{\Theta}) = \clw_{\clp},
\]
by part (2). Part (4) is simply \cite[Corollary 2.3]{JayII}.
\end{proof}

We are now presenting a key property of Beurling tuples, namely that defect operators introduced in Definition \ref{defn: trunc defect} are indeed positive, justifying the superscript $2$ in the definition.

\begin{thm}\label{defect_1}
Let $T \in\mathbb{S}_n^{B}(\clh)$, $j\in I_n$, and let $\clp \subseteq I_n \setminus\{j\}$. Then
\[
 D^2_{j,T,\clp} \geq 0.
\]
\end{thm}
\begin{proof}
As per our understanding (or in view of Theorem \ref{Dilation Thm} and \eqref{eqn: T and CT}), we consider a tuple of model operators $C = (C_1, \ldots, C_n)$ on a quotient module $\clq \subseteq H^2_{\cle_*}(\D^n)$ in place of the Szeg\"{o} tuple $T$ on $\clh$ (here $\cle_*$ is a Hilbert space). For $i\in \clp$, we recall from Lemma \ref{lemma: defect Ci} that $D_{C_j}^2 = P_{\clq} M_{z_j}^* P_{\cls}M_{z_j}|_{\clq}$, where $\clq^\perp = \cls$. Also recall from \eqref{eqn: CjDiCj} that $C_i D_{C_j}^2C_i^* = P_{\clq}M_{z_j}^*P_{z_i\cls} M_{z_j}|_{\clq}$. Since $(I-\Delta_{C_i})(D_{C_j}^2) = D_{C_j}^2-C_iD_{C_j}^2C_i^*$, it follows that
\[
\begin{split}
(I-\Delta_{C_i})(D_{C_j}^2) & = P_{\clq}M_{z_j}^*P_{\cls}M_{z_j}|_{\clq}-P_{\clq}M_{z_j}^*P_{z_i\cls}  M_{z_j}|_{\clq}
\\&=P_{\clq}M_{z_j}^*(P_{\cls}-P_{z_i\cls})M_{z_j}|_{\clq}
\\
&=P_{\clq}M_{z_j}^*P_{\cls \ominus z_i\cls}M_{z_j}|_{\clq}.
\end{split}
\]
In the above, we have used the fact that $\cls$ is a submodule so that $z_i\cls \subseteq \cls$, and hence
\begin{equation}\label{eqn: S minus zS}
P_{\cls}-P_{z_i\cls} = P_{\cls \ominus z_i\cls}.
\end{equation}
Let $k\in \clp$ and $k\neq i$. We proceed to compute
\[
\begin{split}
(I-\Delta_{C_k})(I-\Delta_{C_i})(D_{C_j}^2) & = (D_{C_j}^2-C_iD_{C_j}^2C_i^*) -C_k(D_{C_j}^2-C_iD_{C_j}^2C_i^*)C_k^*
\\
&=P_{\clq}M_{z_j}^*P_{\cls \ominus z_i\cls}M_{z_j}|_{\clq}-P_{\clq}M_{z_k}M_{z_j}^*P_{\cls \ominus z_i\cls}M_{z_j}M_{z_k}^*|_{\clq}
\\
&=P_{\clq}M_{z_j}^* \left(P_{\cls \ominus z_i\cls}- M_{z_k}  P_{\cls \ominus z_i\cls} M_{z_k}^*  \right)  M_{z_j}|_{\clq}
\\
&=P_{\clq}M_{z_j}^*  \left(P_{\cls \ominus z_i\cls}- P_{z_k(\cls \ominus z_i\cls)} \right) M_{z_j}|_{\clq}.
\end{split}
\]
Here we have used the identity, in the spirit of \eqref{eqn: proj onto zS}, that $M_{z_k} P_{\cls \ominus z_i\cls} M_{z_k}^* = P_{z_k(\cls \ominus z_i\cls)}$. By part (2) of Lemma \ref{wandering}, it follows that $P_{\cls \ominus z_i\cls}-  P_{z_k(\cls \ominus z_i\cls)} \geq 0$. In particular, we have
\[
P_{\cls \ominus z_i\cls}-  P_{z_k(\cls \ominus z_i\cls)}=P_{(\cls \ominus z_i\cls)\ominus z_k(\cls \ominus z_i\cls)}=P_{(\cls \ominus z_i\cls)\cap (\cls \ominus z_k\cls)},
\]
and consequently
\[
(I-\Delta_{C_k})(I-\Delta_{C_i})(D_{C_j}^2)=P_{\clq}M_{z_j}^*P_{(\cls \ominus z_i\cls)\cap (\cls \ominus z_k\cls)}M_{z_j}|_{\clq}\geq 0.
\]
Recall that $D^2_{j,C,\clp} = \displaystyle\prod_{i\in \clp}(I-\Delta_{C_i})(D_{C_j}^2)$, and hence, applying induction, one can now say that
\begin{equation}\label{eqn: DjCP}
D^2_{j,C,\clp} = P_{\clq}M_{z_j}^*P_{\clw_{\clp}}M_{z_j}|_{\clq},
\end{equation}
where $\clw_{\clp}= \displaystyle {\bigcap_{i\in \clp}} (\cls \ominus z_i\cls)$. Since $P_{\clq}M_{z_j}^*P_{\clw_{\clp}}M_{z_j}|_{\clq} \geq 0$, we conclude that $D^2_{j,C,\clp} \geq 0$.
\end{proof}

Given the positivity property of truncated defect operators for Beurling tuples as proved above, it is now fitting to talk about positive truncated defect operators and corresponding defect spaces.

\begin{defn}\label{def: trunc def}
Let $T \in\mathbb{S}_n^{B}(\clh)$. For each $j\in I_n$ and $\clp \subseteq I_n \setminus\{j\}$, the truncated defect operator $D_{j,T,\clp}$ is defined by
\[
D_{j,T,\clp} := (D^2_{j,T,\clp})^{\frac{1}{2}}.
\]
The corresponding truncated defect space $\cld_{j,T,\clp}$ is defined by
\[
\cld_{j,T,\clp}=\overline{\text{ran}}D_{j,T,\clp}.
\]
\end{defn}

The proof of the following properties of defect operators and truncated defect spaces follows the same approach as the proof of the preceding theorem.

\begin{cor}\label{cor: defect_1}
Let $T \in\mathbb{S}_n^{B}(\clh)$, $j\in I_n$, and let $\clp \subseteq I_n \setminus\{j\}$. Then
\begin{enumerate}
\item $T_iD^2_{j,T,\clp}T_i^*\leq D^2_{j,T,\clp}$ for all $i \in I_n \setminus  (\clp\cup \{j\})$.			
\item $T_i\cld_{j,T,\clp}\subseteq \cld_{j,T,\clp}$ for all $i\in I_n \setminus  (\clp\cup \{j\})$.
\end{enumerate}
\end{cor}
\begin{proof}
Let $\clp\subsetneq I_n \setminus\{j\}$ and let $i\in I_n \setminus (\clp\cup \{j\})$. In the proof of the above theorem (more specifically, see \eqref{eqn: DjCP}), we have already calculated that $D^2_{j,C,\clp} = P_{\clq}M_{z_j}^*P_{\clw_{\clp}}M_{z_j}|_{\clq}$. Then
\[
\begin{split}
C_iD^2_{j,C,\clp}C^*_i&=P_{\clq}M_{z_i}P_{\clq}M_{z_j}^*P_{\clw_{\clp}}M_{z_j}M_{z_i}^*|_{\clq}
\\
&=P_{\clq}M_{z_i}M_{z_j}^*P_{\clw_{\clp}}M_{z_j}M_{z_i}^*|_{\clq}
\\
&=P_{\clq}M_{z_j}^*M_{z_i}P_{\clw_{\clp}}M_{z_i}^*M_{z_j}|_{\clq}
\\
&=P_{\clq}M_{z_j}^*P_{z_i\clw_{\clp}}M_{z_j}|_{\clq},
\end{split}
\]
and consequently
\[
\begin{split}
D^2_{j,C,\clp}-C_iD^2_{j,C,\clp}C^*_i&=P_{\clq}M_{z_j}^*P_{\clw_{\clp}}M_{z_j}|_{\clq}-P_{\clq}M_{z_j}^*P_{z_i\clw_{\clp}}M_{z_j}|_{\clq}
\\
&=P_{\clq}M_{z_j}^*\left(P_{\clw_{\clp}}-  P_{z_i\clw_{\clp}} \right) M_{z_j}|_{\clq}.
\end{split}
\]
By part (2) of Lemma \ref{wandering}, we know that $P_{\clw_{\clp}}-  P_{z_i\clw_{\clp}} \geq 0$, and hence
\[
D^2_{j,C,\clp}-C_iD^2_{j,C,\clp}C^*_i\geq 0.
\]
This establishes part (1), which specifically implies part (2), thereby completing the proof of the corollary.
\end{proof}

The following order reversing properties of the truncated defect operators and spaces can now be inferred naturally.

\begin{thm}\label{ABC}
Let $T \in\mathbb{S}_n^{B}(\clh)$, $j \in I_n$, and let $\clp_1\subseteq \clp_2 \subseteq I_n \setminus\{j\}$. Then:
\begin{enumerate}
\item[(1)] $D_{j,T,\clp_2}\leq D_{j,T,\clp_1}.$
\item[(2)] $\cld_{j,T,\clp_2}\subseteq \cld_{j,T,\clp_1}.$
\end{enumerate}
\end{thm}
\begin{proof}
Without loss of generality, assume that $\clp_1 \subsetneq \clp_2$. We write $\clp_2=\clp_1 \sqcup \clp_3$, the disjoint union of $\clp_1$ and some nonempty subset $\clp_3$ of $I_n$. Recall that $\Delta_{T_p} \Delta_{T_q} = \Delta_{T_q} \Delta_{T_p}$ for all $p, q \in I_n$. Then
\[
D^2_{j,T,\clp_2} = \prod_{k\in \clp_2}(I-\Delta_{T_k})(D_{T_j}^2) = \prod_{k\in \clp_3}(I-\Delta_{T_k})\prod_{l\in \clp_1}(I-\Delta_{T_l})(D_{T_j}^2) =\prod_{k\in \clp_3}(I-\Delta_{T_k}) (D^2_{j,T,\clp_1}).
\]
If $k\in \clp_3$, then
\[
(I-\Delta_{T_k}) (D^2_{j,T,\clp_1})=D^2_{j,T,\clp_1} - T_kD^2_{j,T,\clp_1}T_k^*\leq D^2_{j,T,\clp_1}.
\]
Using this, by induction, we obtain
\[
D^2_{j,T,\clp_2}=\prod_{k\in \clp_3}(I-\Delta_{T_k}) (D^2_{j,T,\clp_1})\leq D^2_{j,T,\clp_1},
\]
which also shows $\cld_{j,T,\clp_2}\subseteq \cld_{j,T,\clp_1}$, and completes the proof.
\end{proof}

Specifically, we have:
	
\begin{cor}\label{lower bound}
Let $T \in\mathbb{S}_n^{B}(\clh)$. Then $D_{j,T}\leq D_{T_j}$ for all $j \in I_n$.
\end{cor}
\begin{proof}
Set $\clp_1 = \emptyset$ and $\clp_2= I_n \setminus \{j\}$. Then $D^2_{j,T,\clp_1}=D_{T_j}^2$ and   $  D^2_{j,T,\clp_2}=D_{j,T}^2$. The result now directly follows from Theorem \ref{ABC}.
\end{proof}

In particular, if $D_{T_j}=0$ then $D_{j,T}=0$. The converse is also true, which we will address in Corollary \ref{cor: defect conv}.

Next, we fix a notation: Let $T \in \clb(\clh)^n_c$, $j \in I_n$, and let $\clp=\{p_1,\dots,p_r \} \subsetneq I_n \setminus\{j\}$. For each $\alpha=(\alpha_1\dots,\alpha_r)\in \mathbb{Z}^r_+$, we define
\[
T^{\alpha}_{\clp}:=T_{p_1}^{\alpha_1}\dots T_{p_r}^{\alpha_r},
\]
whenever $\clp \neq \emptyset$, and $T^{\alpha}_{\clp} = I_\clh$ if $\clp = \emptyset$. For all $j \in I_n$, we recall from \eqref{eqn: D jT} that
\[
D^2_{j,T} = \prod_{k \in I_n \setminus \{j\}} (I_{\clb(\clh)} - \Delta_{T_k})(D^2_{T_j}),
\]
and by Definition \ref{def: trunc def}, we have $\cld_{j,T}=\overline{\text{ran}}D_{j,T}$. With this notation, we present series expansions of squares of truncated defect operators. These are fundamental in representing commuting tuples of contractions.

\begin{prop}\label{generating sets}
Let $T \in \mathbb{S}_n^{B}(\clh)$, $j \in I_n$, and let $\clp \subsetneq I_n \setminus\{j\}$. Then
\begin{enumerate}
\item $D_{T_j}^2 = \displaystyle \sum_{\alpha\in \mathbb{Z}_+^{|\clp|}}T^{\alpha}_{\clp}D_{j,T,\clp}^2T^{*\alpha}_{\clp}$, and
\item $D^2_{j,T, \clp} = \displaystyle \sum_{\alpha\in\mathbb{Z}_+^{|\clp^{\prime}|}} T^{\alpha}_{\clp^{\prime}} D_{j,T}^2 T^{*\alpha}_{\clp^{\prime}}$, where $\clp^{\prime}= I_n \setminus (\clp\cup \{j\})$.
\end{enumerate}
\end{prop}
\begin{proof}
Recall, by definition, that $D^2_{j,T,\clp} = \displaystyle\prod_{i\in \clp}(I-\Delta_{T_i})  (D^2_{T_j})$. We prove $(1)$ by induction on $|\clp|$. If $\clp=\emptyset$, then nothing to prove. Suppose $\clp=\{i\}$ (so that $|\clp|=1$). Then
\[
D^2_{j,T,\clp}=(I-\Delta_{T_{i}})(D_{T_j}^2)=D_{T_j}^2- T_{i}D_{T_j}^2T_{i}^*,
\]
implies $D_{T_j}^2=D^2_{j,T,\clp}+T_{i}D_{T_j}^2T_{i}^*$. Repeating this identity results in the equality 
\[
D_{T_j}^2=\sum_{k\in \mathbb{Z}_+}T^k_{i}D_{j,T,\clp}^2T_{i}^{*k}=\sum_{k\in \mathbb{Z}^{|\clp|}_+}T^k_{\clp}D_{j,T,\clp}^2T_{\clp}^{*k}.
\]
As $T_i \in C_{\cdot 0}$, above series converges in the strong operator topology. Thus (1) is true for any $\clp$ with $|\clp|=1$. Now let $|\clp|=2$ and assume that $\clp=\{i_1,i_2\}$. In this case, we have
\[
\begin{split}
D^2_{j,T,\clp} & =(I-\Delta_{T_{i_2}})(I-\Delta_{T_{i_1}})(D_{T_j}^2)
\\
&=(I-\Delta_{T_{i_2}})(D_{T_j}^2)- (I-\Delta_{T_{i_2}})(T_{i_1}D_{T_j}^2T_{i_1}^*)
\\
&=(I-\Delta_{T_{i_2}})(D_{T_j}^2)- T_{i_1}\left( (I-\Delta_{T_{i_2}})(D_{T_j}^2)  \right)T_{i_1}^*.
\end{split}
\]
This shows that
\[
D^2_{j,T,\clp_1}=(I-\Delta_{T_{i_2}})(D_{T_j}^2)=  \sum_{k\in \mathbb{Z}_+}T^k_{i_1} 	D^2_{j,T,\clp}T^{*k}_{i_1},
\]
where $\clp_1=\{i_2\}$. As $|\clp_1|=1$, we get
\[
D_{T_j}^2=\sum_{k\in \mathbb{Z}^{|\clp_1|}_+}T^k_{\clp_1}D_{j,T,\clp_1}^2T_{\clp_1}^{*k}=\sum_{k\in \mathbb{Z}^{|\clp_1|}_+}    T^k_{\clp_1}  \left(\sum_{k\in \mathbb{Z}_+}  T^k_{i_1} 	D^2_{j,T,\clp}T^{*k}_{i_1}\right) T_{\clp_1}^{*k}=\sum_{k\in \mathbb{Z}^{|\clp|}_+}T^k_{\clp}D_{j,T,\clp}^2T_{\clp}^{*k},
\]
proving (2) for any $\clp$ with $|\clp|=2$. Finally, assume that (2) is true for any $\clp$ such that $|\clp|=l$. Let $\clp_+=\clp\cup \{r\}\subseteq I_n \setminus\{j\}$ and $|\clp_+|=l+1$.  Then
\[
\begin{split}
D^2_{j,T,\clp_+} & =  \prod_{i\in \clp_+}(I-\Delta_{T_i}) (D^2_{T_j})
\\
&=\prod_{i\in \clp}(I-\Delta_{T_i}) (I-\Delta_{T_r}) (D^2_{T_j})
\\
&=\prod_{i\in \clp}(I-\Delta_{T_i})(D^2_{T_j})-  T_r \left(\prod_{i\in \clp}(I-\Delta_{T_i})(D^2_{T_j})\right)T_r^*
\\
&=D_{j,T,\clp}^2-T_rD_{j,T,\clp}^2T_r^*,
\end{split}
\]
that is, $D_{j,T,\clp}^2=D^2_{j,T,\clp_+}+T_rD_{j,T,\clp}^2T_r^*$, and hence
\[
D_{j,T,\clp}^2=  \sum_{k\in \mathbb{Z}_+} T_r^k D^2_{j,T,\clp_+}T_r^{*k}.
\]
Based on our assumption, we have
\[
D_{T_j}^2 = \sum_{\alpha\in \mathbb{Z}_+^{|\clp|}}T^{\alpha}_{\clp}D_{j,T,\clp}^2T^{*\alpha}_{\clp} = \sum_{\alpha\in \mathbb{Z}_+^{|\clp|}}T^{\alpha}_{\clp}\left(\sum_{k\in \mathbb{Z}_+} T_r^k D^2_{j,T,\clp_+}T_r^{*k}\right)T^{*\alpha}_{\clp} = \sum_{\alpha\in \mathbb{Z}_+^{|\clp_+|}}T^{\alpha}_{\clp_+}D_{j,T,\clp_+}^2T^{*\alpha}_{\clp_+}.
\]
Thus (1) is true for any $\clp$ such that $|\clp|=l+1$. This completes the proof of part (1). For  part (2), consider  $I_n = \clp\cup \clp^{\prime}\cup \{j\}$. Then
\[
\begin{split}
D^2_{j,T}&=\prod_{m \in \{j\}^c} (I-\Delta_{T_m})(D_{T_j}^2) =\prod_{k\in \clp^{\prime}}(I-\Delta_{T_k})\prod_{l\in \clp} (I-\Delta_{T_l})(D_{T_j}^2) =\prod_{k\in \clp^{\prime}}(I-\Delta_{T_k})(D^2_{j,T, \clp}),
\end{split}
\]
and hence
\[
\sum_{\alpha\in\mathbb{Z}_+^{|\clp^{\prime}|}} T^{\alpha}_{\clp^{\prime}} D_{j,T}^2 T^{*\alpha}_{\clp^{\prime}}=\sum_{\alpha\in\mathbb{Z}_+^{|\clp^{\prime}|}  } T^{\alpha}_{\clp^{\prime}} \left(\prod_{k\in \clp^{\prime}}(I-\Delta_{T_k})(D^2_{j,T, \clp})\right) T^{*\alpha}_{\clp^{\prime}}=D^2_{j,T, \clp},
\]
which completes the proof of part (2) and the proposition.
\end{proof}

With the help of the above, we can now write down representations of truncated defect spaces of Beurling tuples.

\begin{cor}\label{more generating sets}
Let $T \in \mathbb{S}_n^{B}(\clh)$, $j \in I_n$, and let $\clp \subseteq I_n \setminus \{j\}$. Then
\begin{enumerate}
\item $\cld_{T_j} = \displaystyle \bigvee_{\alpha\in \mathbb{Z}_+^{|\clp|}}T^{\alpha}_{\clp}\cld_{j,T,\clp}$, and
\item $\cld_{j,T,\clp} = \displaystyle \bigvee_{\alpha\in \mathbb{Z}^{|\clp^{\prime}|}_+}  T^{\alpha}_{\clp^{\prime}} \cld_{j,T}$, where $\clp^{\prime}= I_n \setminus (\clp\cup \{j\})$.
\end{enumerate}
\end{cor}	
\begin{proof}
By Proposition \ref{generating sets}, we know that
\[
D_{T_j}^2=\sum_{\alpha\in \mathbb{Z}_+^{|\clp|}}T^{\alpha}_{\clp}D_{j,T,\clp}^2T^{*\alpha}_{\clp}\geq T^{\alpha}_{\clp}D_{j,T,\clp}^2T^{*\alpha}_{\clp},
\]
which, in particular, guarantees $T^{\alpha}_{\clp^{\prime}} \cld_{j,T}\subseteq \cld_{T_j}$ for any $\alpha\in \mathbb{Z}_+^{|\clp|}$, and consequently
\[
\bigvee_{\alpha\in \mathbb{Z}_+^{|\clp|}}T^{\alpha}_{\clp}\cld_{j,T,\clp}\subseteq \cld_{T_j}.
\]
For the reverse inclusion, let $h\in \cld_{T_j}$ and suppose $h\perp T^{\alpha}_{\clp}\cld_{j,T,\clp}$ for all $\alpha\in \mathbb{Z}_+^{|\clp|}$. Equivalently, $ T^{*\alpha}_{\clp}h\perp \cld_{j,T,\clp}$ for all $\alpha\in \mathbb{Z}_+^{|\clp|}$. So
\[
D_{j,T,\clp}T^{*\alpha}_{\clp}h=0,
\]
for all  $\alpha\in \mathbb{Z}_+^{|\clp|}$. By this and Proposition \ref{generating sets}, we get
\[
\|D_{T_j}h\|^2 = \sum_{\alpha\in \mathbb{Z}_+^{|\clp|}}\|D_{j,T,\clp}T^{*\alpha}_{\clp}h\|^2 = 0,
\]
and hence $h\in \cld_{T_j} \cap \ker D_{T_j} =\{0\}$. This shows that $\cld_{T_j}= \displaystyle \vee_{\alpha\in \mathbb{Z}_+^{|\clp|}}T^{\alpha}_{\clp}\cld_{j,T,\clp}$. The remaining part can be proved in a similar way.
\end{proof}

In particular, when $\clp = I_n\setminus \{j\}$, we have the following useful representations of defect spaces:
	
\begin{cor}\label{span_defect}
Let $T \in\mathbb{S}_n^{B}(\clh)$ and let $j \in I_n$. Then
\[
\cld_{T_j} = \bigvee_{\alpha\in \mathbb{Z}_+^{n-1}} T^{\alpha}_{I_n \setminus \{j\}} \cld_{j,T}.
\]
\end{cor}

 
As a result, the converse of Corollary \ref{lower bound} also holds good.
 	
\begin{cor}\label{cor: defect conv}
Let $T \in\mathbb{S}_n^{B}(\clh)$. Then $D_{T_j}=0$ if and only if $D_{j,T}=0$.
\end{cor}

Recall that if $T \in\mathbb{S}_n^{B}(\clh)$, then $D_{T_i} D_{T_j} = 0$ for all $i \neq j$. We have in addition that:

\begin{cor}\label{product_0}
If $T \in\mathbb{S}_n^{B}(\clh)$, then $D_{i,T} D_{j,T}=0$ for all $i\neq j$.
\end{cor}
\begin{proof}
By Corollary \ref{span_defect}, we know that $\cld_{i,T}\subseteq \cld_{T_i}$ and $\cld_{j,T}\subseteq \cld_{T_j}$. Hence by Theroem \ref{Beurling_qm}, the result follows.
\end{proof}

As pointed out, truncated defect operators will be placed at the diagonal of the defect operators of the second kind. For the off-diagonal entries, in the following section, we will construct a new notion of commutators.

\section{Joint commutators}\label{sec: joint comm}

In this section, we introduce a new idea of commutators that will be parameterized by subsets of $I_n$ of a cardinality two. This is part of the search for suitable defect operators and spaces (of the second kind) of Beurling tuples. Recall that the commutator of two bounded linear operators $A$ and $B$ on $\clh$ is defined by
\[
[A,B]=AB-BA.
\]
We now come to yet another key notion that appears to bear some more significance in multivariable operator theory. 
 

\begin{defn}\label{def: joint comm}
Let $T \in \clb(\clh)^n_c$. For each $i, j \in I_n$, $n > 2$, with $i\neq j$, the $(i,j)$-th joint commutator of $T$ is defined by
\[
\delta_{ij}(T) = \prod_{k \in \{i,j\}^c}(I-\Delta_{T_k})([T_j,T_i^*]).
\]
If $n=2$, then we set $\delta_{ij}(T) = [T_j,T_i^*]$. We will simply say this as a joint commutator whenever $i$ and $j$ are clear from the context.
\end{defn}

We aim to prove that the function $\delta_{ij}: \clb(\clh) \raro \clb(\clh)$ involves wandering subspaces. Recall that for a submodule $\cls \subseteq H^2_{\cle_*}(\D^n)$, the wandering subspace $\clw_\clp$, for all nonempty $\clp \subseteq I_n$, is defined by
\begin{equation}\label{eqn: W_P}
\clw_{\clp}=\bigcap_{i\in \clp}(\cls \ominus z_i\cls).
\end{equation}
We also shorten the notation by expressing $\clw = \clw_{I_n}$, and
\[
\clw_{\{j\}^c} = \clw_{I_n \setminus \{j\}} \qquad (j \in I_n).
\]

\begin{prop}\label{corss_commutator}
Let $n \geq 2$, and let $\clq$ be a Beurling quotient module of $H^2_{\cle_*}(\D^n)$.
Then
\[
[C_j,C_i^*] = P_\clq M_{z_i}^* P_\cls M_{z_j}|_\clq,
\]
and
\[
\delta_{ij}(C) = P_{\clq}M_{z_i}^*P_{\clw_{\clp_{i,j}}}M_{z_j}|_{\clq}.
\]
for all $i \neq j$, where $\clp_{i,j} = I_n \setminus \{i,j\}$.
\end{prop}	

\begin{proof}
Pick $i, j \in I_n$, and assume that $i \neq j$. We begin simplifying the commutator. In view of $P_\clq M_{z_j} P_\clq = P_\clq M_{z_j}$ and $M_{z_j} M_{z_i}^* = M_{z_i}^* M_{z_j} M_{z_i}^*$, we have
\[
\begin{split}
[C_j,C_i^*] & = C_j C_i^* - C_i^* C_j
\\
& = P_\clq M_{z_j} P_\clq M_{z_i}^*|_\clq - P_\clq M_{z_i}^* P_\clq M_{z_j}|_\clq\\
& = P_\clq M_{z_j} M_{z_i}^*|_\clq - P_\clq M_{z_i}^* M_{z_j}|_\clq + P_\clq M_{z_i}^* P_\cls M_{z_j}|_\clq
\\
& = P_\clq M_{z_i}^* P_\cls M_{z_j}|_\clq.
\end{split}
\]
Now, we pick $k \in I_n$, and assume that $k\neq i,j$. Then
\[
\begin{split}
(I-\Delta_{C_k})([C_j,C_i^*])&=[C_j,C_i^*]-C_k[C_j,C_i^*]C_k^*
\\
&=P_{\clq}M_{z_i}^*P_{\cls}M_{z_j}|_{\clq}-P_{\clq}M_{z_k}P_{\clq}M_{z_i}^*P_{\cls}M_{z_j}M_{z_k}^*|_{\clq}
\\
&=P_{\clq}M_{z_i}^*P_{\cls}M_{z_j}|_{\clq}-P_{\clq}M_{z_k} (I-P_\cls)M_{z_i}^*P_{\cls}M_{z_j}M_{z_k}^*|_{\clq}
\\
&=P_{\clq}M_{z_i}^*P_{\cls}M_{z_j}|_{\clq}-P_{\clq}M_{z_k}M_{z_i}^*P_{\cls}M_{z_j}M_{z_k}^*|_{\clq},
\end{split}
\]
as $P_\clq M_{z_k} P_\cls = 0$. This further simplifies our computation to
\[
\begin{split}
(I-\Delta_{C_k})([C_j,C_i^*])& = P_{\clq}M_{z_i}^*P_{\cls}M_{z_j}|_{\clq}-P_{\clq} M_{z_i}^*M_{z_k} P_{\cls} M_{z_k}^*M_{z_j}|_{\clq}
\\
&=P_{\clq}M_{z_i}^*P_{\cls}M_{z_j}|_{\clq}-P_{\clq}M_{z_i}^*P_{z_k\cls}M_{z_j}|_{\clq}
\\
&=P_{\clq}M_{z_i}^*(P_{\cls}-P_{z_k\cls})M_{z_j}|_{\clq}
\\
&=P_{\clq}M_{z_i}^*P_{\cls\ominus z_k\cls}M_{z_j}|_{\clq}.
\end{split}
\]
In the above, we have used the fact that $M_{z_k} P_{\cls} M_{z_k}^*$ is a projection and $M_{z_k} P_{\cls} M_{z_k}^* = P_{z_k \cls}$, and $P_{\cls}-P_{z_k\cls} = P_{\cls\ominus z_k\cls}$ (see \eqref{eqn: S minus zS}). Let $m\in I_n \setminus \{i,j,k\}$. Then
\[
\begin{split}
(I-\Delta_{C_m})(I-\Delta_{C_k})([C_j,C_i^*]) & = (I- \Delta_{C_m})(P_{\clq}M_{z_i}^*P_{\cls\ominus z_k\cls}M_{z_j}|_{\clq})
\\
&=P_{\clq}M_{z_i}^*P_{\cls\ominus z_k\cls}M_{z_j}|_{\clq}-P_{\clq}M_{z_m}P_{\clq}M_{z_i}^*P_{\cls\ominus z_k\cls}M_{z_j}M_{z_m}^*|_{\clq}
\\
&=P_{\clq}M_{z_i}^*P_{\cls\ominus z_k\cls}M_{z_j}|_{\clq}-P_{\clq}M_{z_m}M_{z_i}^*P_{\cls\ominus z_k\cls}M_{z_j}M_{z_m}^*|_{\clq}
\\
&=P_{\clq}M_{z_i}^*P_{\cls\ominus z_k\cls}M_{z_j}|_{\clq}-P_{\clq}M_{z_i}^*(M_{z_m}P_{\cls\ominus z_k\cls}M_{z_m}^*)M_{z_j}|_{\clq}.
\end{split}
\]
But again, $M_{z_m}P_{\cls\ominus z_k\cls}M_{z_m}^*$ is a projection, and $M_{z_m}P_{\cls\ominus z_k\cls}M_{z_m}^* = P_{(\cls\ominus z_k\cls)\ominus z_m(\cls\ominus z_k\cls)}$, and hence
\[
\begin{split}
(I-\Delta_{C_m})(I-\Delta_{C_k})([C_j,C_i^*]) & =P_{\clq}M_{z_i}^*(P_{\cls\ominus z_k\cls}-P_{z_m(\cls\ominus z_k\cls)})M_{z_j}|_{\clq}
\\
&=P_{\clq}M_{z_i}^*P_{(\cls\ominus z_k\cls)\ominus z_m(\cls\ominus z_k\cls)}M_{z_j}|_{\clq}
\\
&=P_{\clq}M_{z_i}^*P_{(\cls\ominus z_k\cls)\cap (\cls\ominus z_m\cls)}M_{z_j}|_{\clq}.
\end{split}
\]
Continuing this line of arguments, we finally conclude that
\[
\begin{split}
\delta_{ij}(C) = P_{\clq}M_{z_i}^*P_{\clw_{\clp_{i,j}}}M_{z_j}|_{\clq} = P_{\clq}M_{z_i}^*P_{\clw_{\clp_{i,j}}}M_{z_j}|_{\clq}.
\end{split}
\]	
where  $\clp_{i,j}= I_n \setminus \{i,j\}$. This completes the proof.
\end{proof}

We use Lemma \ref{wandering} to prove the following useful identity. The proof does not use the Beurling inner functions directly, and so we will keep it suppressed.

\begin{lem}\label{W=Wj}
Let $\clq$ be a Beurling quotient module of $H^2_{\cle_*}(\D^n)$ and let $j \in I_n$. Then
\[	
P_{\clw}M_{z_j}|_{\clq}=P_{\clw_{\{j\}^c}}M_{z_j}|_{\clq}.
\]
\end{lem}
\begin{proof}
By part (4) of Lemma \ref{wandering}, we get
\[
P_\clw = P_{\clw_{\{j\}^c} \ominus z_j \clw_{\{j\}^c}} = P_{\clw_{\{j\}^c}} - P_{z_j \clw_{\{j\}^c}}.
\]
As in \eqref{eqn: proj onto zS} that $P_{z_j \clw_{\{j\}^c}} = M_{z_j} P_{\clw_{\{j\}^c}} M_{z_j}^*$, and hence
\[
P_{\clw}M_{z_j}|_{\clq} = P_{\clw_{\{j\}^c}}M_{z_j}|_{\clq} - (M_{z_j}P_{\clw_{\{j\}^c}} M_{z_j}^*)M_{z_j}|_{\clq} = P_{\clw_{\{j\}^c}}M_{z_j}|_{\clq}-M_{z_j}P_{\clw_{\{j\}^c}}|_{\clq}.
\]
Since $P_{\clw_{\{j\}^c}}|_{\clq} = 0$, the result follows immediately.
\end{proof}

Now we apply this to joint commutators of Szeg\"{o} tuples. Recall, for each $j \in I_n$, that $\cld_{j,T} = \overline{\text{ran}}D_{j,T}$, where (see \eqref{eqn: D jT})
\[
D_{j,T}^2 = \Big(\prod_{k \in I_n\setminus \{j\}} (I_{\clb(\clh)} - \Delta_{T_k}) \Big)(D^2_{T_j}).
\]

\begin{prop}\label{corss commutator}
Let $T \in \mathbb{S}_n^B(\clh)$, $i, j \in I_n$, and suppose $i\neq j$. Then
\[
\overline{ran}\delta_{ij}(T)\subseteq \cld_{i,T}.
\]
\end{prop}
\begin{proof}
We consider $T$ as a tuple of model operators $C = (C_1,\dots, C_n)$ on some Beurling quotient module $\clq \subseteq H^2_{\cle_*}(\D^n)$. By Theorem \ref{corss_commutator}, we know that $\delta_{ij}(C) = P_{\clq}M_{z_i}^*P_{\clw_{\clp_{i,j}}}M_{z_j}P_{\clq}$, where $\clp_{i,j} = I_n \setminus \{i,j\}$. By part (4) of Lemma \ref{wandering}, we also know that $\clw_{\clp_{i,j}}\ominus z_j \clw_{\clp_{i,j}} = {\clw_{\{i\}^c}}$. Writing $P_{\clw_{\{i\}^c}} = P_{\clw_{\clp_{i,j}}} - P_{z_j \clw_{\clp_{i,j}}}$, this says
\[
\delta_{ij}(C) = P_{\clq}M_{z_i}^*P_{\clw_{\clp_{i,j}}}M_{z_j}P_{\clq} = P_{\clq}M_{z_i}^* (P_{\clw_{\clp_{i,j}} }-M_{z_j} P_{\clw_{\clp_{i,j}} } M_{z_j}^*)M_{z_j}P_{\clq} = P_{\clq}M_{z_i}^*P_{\clw_{\{i\}^c}} M_{z_j}P_{\clq}.
\]
In the above, we used $P_{\clw_{\clp_{i,j}}} P_{\clq} = 0$. Applying Lemma \ref{W=Wj} to this yields
\begin{equation}\label{crosscommmmmmmmm}
\delta_{ij}(C) = P_{\clq}M_{z_i}^*P_{\clw}M_{z_j}P_{\clq},
\end{equation}
which implies $\overline{ran}\delta_{ij}(C) \subseteq \text{clos}(P_{\clq}M_{z_i}^* \clw)$. By \eqref{eqn: DjCP} and Lemma \ref{W=Wj}, we get
\[
D_{i,C}^2=P_{\clq}M_{z_i}^*P_{\clw_{\{i\}^c}} M_{z_i}|_{\clq}=P_{\clq}M_{z_i}^*P_{\clw}M_{z_i}|_{\clq},
\]
which means that
\begin{equation}\label{eqn: clDiC}
\cld_{i,C}=\text{clos}(P_{\clq}M_{z_i}^*(\clw)),
\end{equation}
and hence, we conclude that $\overline{ran}\delta_{ij}(C)\subseteq \cld_{i,C}$.
\end{proof}
	
As a consequence, we have:
\[
D_{k,T} \delta_{ij}(T) =0.
\]
for all $k\neq i$ and $i\neq j$. Indeed, by Corollary \ref{product_0}, we know that $D_{k,T}D_{i,T}=0$. For each $f\in \clq$, using Proposition \ref{corss commutator}, we find
\[
D_{k,T}\delta_{ij}(T) f \in D_{k,T}\cld_{i,T}=\{0\},
\]
which shows that $D_{k,T} \delta_{ij}(T) = 0$.

We now revisit Beurling quotient modules and analyze some of their essential properties with respect to truncated defect spaces.

\begin{prop}\label{Structure}
Consider the tuple of model operators $C$ on $\clq_\Theta$, where $\Theta\in H^{\infty}_{\clb(\cle,\cle_*)}(\D^n)$ is an inner function. Given $j \in I_n$, we have the following:
\begin{enumerate}
\item $\cld_{j,C}= \text{clos}(M_{z_j}^*\Theta \cle)$, and
\item $\cld_{C_j} = \displaystyle \bigvee_{\alpha\in \mathbb{Z}^{|\clp|}_+} M_{z,\clp}^{\alpha} M_{z_j}^*\Theta \cle$, where $\clp= I_n \setminus \{j\}$.
\end{enumerate}
\end{prop}
\begin{proof}
It is easy to see that $M_{z_j}^*(\Theta \cle) \perp \Theta H^2_{\cle_*}(\D^n)$, and so $M_{z_j}^*(\Theta \cle)\subseteq \clq_{\Theta}$. By \eqref{eqn: clDiC}, we have
\[
\cld_{j,C} = \text{clos}(P_{\clq_{\Theta}}M_{z_j}^*\clw_{\{j\}^c}) = \overline{\text{ran}}(P_{\clq_{\Theta}}M_{z_j}^*|_{\clw_{\{j\}^c}}).
\]
By Lemma \ref{W=Wj}, we get $P_{\clq_{\Theta}}M_{z_j}^*|_{\clw_{\{j\}^c}}=P_{\clq_{\Theta}}M_{z_j}^*|_{\clw}$. Now an application of Lemma \ref{wandering} shows that
\[
\cld_{j,C} = \text{clos}(P_{\clq_{\Theta}}M_{z_j}^*\Theta\cle)=\text{clos}(M_{z_j}^*\Theta\cle),
\]
as $M_{z_j}^*(\Theta \cle)\subseteq \clq_{\Theta}$. This proves (1). For (2), by Lemma \ref{span_defect} again,
\[
\cld_{C_j}=\bigvee_{\alpha\in \mathbb{Z}^{|\clp|}_+} M_{z,\clp}^{\alpha}\cld_{j,C} =\bigvee_{\alpha\in \mathbb{Z}^{|\clp|}_+} M_{z,\clp}^{\alpha} (M_{z_j}^*\Theta \cle),
\]
which completes the proof of the proposition.
\end{proof}

Recall that a quotient module $\clq$ of $H^2_{\cle_*}(\D^n)$ is minimal if
\[
\bigvee_{\alpha\in \mathbb{Z}^n_+}z^{\alpha}\clq = H^2_{\cle_*}(\D^n).
\]
This is in view of the fact that $C$ on $\clq$ dilates to $(M_{z_1}, \ldots, M_{z_n})$ on $H^2_{\cle_*}(\D^n)$ (see Theorem \ref{Dilation Thm}). Now we prove that, with the minimality assumption for Beurling quotient modules, the corresponding submodules do not intersect the fiber spaces.

\begin{lem}\label{minimal}
Let $\Theta \in H^\infty_{\clb(\cle, \cle_*)}(\D^n)$ be an inner function. If $\clq_{\Theta}\subseteq H^2_{\cle_*}(\D^n)$ is a minimal quotient module if and only if
\[
\cls_{\Theta}\cap \cle_*=\{0\}.
\]
\end{lem}
\begin{proof}
Suppose that $h\in \cls_{\Theta}\cap \cle_*$.
We want to show that $h=0$. First, we claim that $h\in \clw$: There exist $f \in H^2_\cle(\D^n)$ and $\zeta \in \cle_*$ such that $h = \Theta f = \zeta$. Then
\[
f = M_\Theta^* \zeta = \Theta(0)^* \zeta:= \eta \in \cle,
\]
implying that $h = \Theta \eta \in \Theta \cle = \clw$. Now, recall that part (1) of Lemma \ref{wandering} implies $\clw = \Theta \cle$. This proves the claim along with the fact that $h=\Theta \eta \in \clw$ for some $\eta \in \cle$. We now note that $h=0$ if and only if $\eta=0$. If possible, suppose that $h \neq 0$. Since, $h\in \cle_*$, we write
\[
\cle_*=\cle_0\oplus \cle_1,
\]
where $\cle_0=\mathbb{C} h$. Then, $H^2_{\cle_*}(\D^n)=H^2_{\cle_0}(\D^n)\oplus H^2_{\cle_1}(\D^n)$. Since $z^{\alpha}h\in \cls_{\Theta}$ for all $\alpha\in \mathbb{Z}^n_+$, it follows that
\[
H^2_{\cle_0}(\D^n)=\bigvee_{\alpha\in \mathbb{Z}^n_+}z^{\alpha}h \subseteq \cls_{\Theta}.
\]
This implies
\[
\clq_{\Theta}=\cls_{\Theta}^{\perp} \subseteq \big(H^2_{\cle_0}(\D^n)\big)^{\perp}=H^2_{\cle_*}(\D^n)\ominus H^2_{\cle_0}(\D^n)=H^2_{\cle_1}(\D^n).
\]
As $\clq_{\Theta}$ is minimal in $H^2_{\cle_*}(\D^n)$, we obation
\[
H^2_{\cle_*}(\D^n)=\bigvee_{\alpha\in \mathbb{Z}^n_+}z^{\alpha} \clq_{\Theta}
\subseteq H^2_{\cle_1}(\D^n)\subsetneq H^2_{\cle_*}(\D^n).
\]
This contradiction leads us to conclude that $h=0$. 

\NI For the reverse direction, if possible, assume that $\clq_{\Theta}$ is not minimal. Then there exist closed subspaces $\clf$ and $\clg$ of $\cle_*$ such that $\cle_*=\clf\oplus \clg$, $\dim \clg\geq 1$, and 
\[
\bigvee_{\alpha\in \mathbb{Z}_+^n} z^{\alpha}\clq_{\Theta}=H^2_{\clf}(\D^n)\subseteq H^2_{\cle_*}(\D^n).
\]
Let $0\neq\eta\in \clg$. As $\eta\perp \clq_{\Theta}$, it follows that $\eta \in \cls_{\Theta}$, and hence $\eta\in \cls_{\Theta}\cap \cle_*$. This yields a contradiction to the fact that $\cls_{\Theta}\cap \cle_*=\{0\}$.
\end{proof}

We now proceed to compute the wandering subspaces of Beurling submodules (see \eqref{eqn: W_P}).

\begin{prop}\label{wandering_defect}
Let $\Theta \in H^\infty_{\clb(\cle, \cle_*)}(\D^n)$ be an inner function. If $\clq_{\Theta}\subseteq H^2_{\cle_*}(\D^n)$ is a minimal quotient module, then
\[
\clw=\bigvee_{j \in I_n} P_{\clw_{\{j\}^c}}M_{z_j}\clq_{\Theta}=\bigvee_{j \in I_n} P_{\clw}M_{z_j}\cld_{j,C} = \bigvee_{j \in I_n} P_{\clw}M_{z_j}\clq_{\Theta} = \bigvee_{j \in I_n} P_{\clw_{\{j\}^c}} M_{z_j}\cld_{j,C}.
\]
\end{prop}
\begin{proof} 		
It is evident that $P_{\{j\}^c} M_{z_j}\clq_{\Theta}\subseteq \clw_{\{j\}^c}$. Let $f \in \clq_{\Theta}$ and $h\in \clw_{\{j\}^c}$. Part (3) of Lemma \ref{wandering} tells us that $z_j \clw_{\{j\}^c} \subseteq \clw_{\{j\}^c}$, that is, $P_{\clw_{\{j\}^c}} z_jh = z_jh$, and hence
\[
\la P_{\clw_j}M_{z_j}f, z_jh \ra = \la z_j f, P_{\clw_j} z_jh \ra = \la z_j f, z_jh \ra = \la f, h  \ra=0,
\]
which implies $P_{\clw_{\{j\}^c}}M_{z_j}\clq_{\Theta}\perp z_j\clw_{\{j\}^c}$. So, $P_{\clw_{\{j\}^c}} M_{z_j}\clq_{\Theta}\subseteq \clw_{\{j\}^c} \ominus z_j \clw_{\{j\}^c}$. However, part (4) of Lemma \ref{wandering} says $\clw= \clw_{\{j\}^c} \ominus z_j \clw_{\{j\}^c}$, and consequently
\[
P_{\clw_{\{j\}^c}}M_{z_j}\clq_{\Theta}\subseteq \clw
\]
for all $j \in I_n$. This gives the one-sided containment that
\[
\bigvee_{j=1}^n P_{\{j\}^c} M_{z_j}\clq_{\Theta}\subseteq \clw.
\]
To prove the reverse containment, pick $h\in \clw$ and assume that $h\perp P_{\clw_{\{j\}^c}}M_{z_j}\clq_{\Theta}$ for all $j$. Then $M_{z_j}^*h\perp \clq_{\Theta}$ for all $j \in I_n$. As $M_{z_j}^*h\in \clq_{\Theta}$, we get $M_{z_j}^*h=0$ for all $j \in I_n$. This implies $h$ is a constant function (that is, $h \in \cle_*$). By Lemma \ref{minimal}, $h=0$, and hence
\[
\clw=\bigvee_{j \in I_n} P_{\clw_{\{j\}^c}} M_{z_j}\clq_{\Theta}.
\]
Let $f\in P_{\clw}M_{z_j}\clq_{\Theta}$ and $f\perp P_{\clw}M_{z_j}\cld_{j,C}$. Then $M_{z_j}^*f\perp \cld_{j, C}$. By (1) of Proposition \ref{Structure}, we have $\cld_{j,C} = \overline{M_{z_j}^*\clw}$ (note that $\clw = \Theta \cle$), and hence $P_{\clq} M_{z_j}^* P_\clw = P_{\cld_{j,C}} M_{z_j}^* P_\clw$. By taking adjoint, we have $P_\clw M_{z_j} P_\clq = P_\clw M_{z_j} P_{\cld_{j,C}}$, which yields
\[
\text{clos} (P_{\clw}M_{z_j}\clq_{\Theta})= \text{clos}  ( P_{\clw} M_{z_j} \cld_{j,C}).
\]
From Lemma \ref{W=Wj}, we get $\text{clos}(P_{\clw}M_{z_j}\clq_{\Theta})=\text{clos}(P_{\clw_{\{j\}^c}} M_{z_j}\clq_{\Theta})$ and
\[
P_{\clw}M_{z_j}|_{\cld_{j,C}}=P_{\clw_{\{j\}^c}}M_{z_j}|_{\cld_{j,C}}.
\]
Hence, $\text{clos}(P_{\clw}M_{z_j}\cld_{j,C}) = \text{clos}(P_{\clw_{\{j\}^c}}M_{z_j}\cld_{j,C})$. This completes the proof.
\end{proof}


We will soon observe a connection between the wandering subspace $\clw$ and the defect space of the second kind that we are attempting to construct. The notion of commutators in multivariable operator theory is an unresolved issue. In Subsection \ref{sub: new comm}, we introduce yet another notion of commutator.

\section{Joint defect operators}\label{sec: defect op}

The characteristic function of a single contraction $T$ is canonically associated with pairs of defect spaces, notably $D_{T^*}$ and $D_T$. In the context of Szeg\"{o} tuples, the canonical dilations clearly highlight the first defect operator (that is, the defect operator of the first kind). This is particularly $D_{T^*} = (\mathbb{S}_n^{-1}(T, T^*))^{\frac{1}{2}}$. The predominant problem, however, lies in recognizing the defect operator of the second kind, which we aim to address in this section. The construction of the other defect is complex, but in the end, it will come out to be explicit and natural. 

Given a Hilbert space $\clh$, we denote by $\clh^n$ the orthogonal direct sum of $n$-copies of $\clh$:
\[
\clh^n = \underbrace{\clh \oplus \cdots \oplus \clh}_{n-\text{times}}.
\]
The defect operators of the first kind in the context of a commuting tuple of contractions use both the notions of truncated defect operators (see \eqref{eqn: D jT}) and joint commutators introduced in Definition \ref{def: joint comm}. A priori, in the following, the superscript $2$ attached to the defect operator does not mean to guarantee that it is positive. However, we will soon clarify that, for Beurling tuples, it is indeed a positive operator.

\begin{defn}\label{defect}
Let $T \in \clb_c^n(\clh)$. The defect operator of the second kind, or joint defect operator of $T$ is defined by
\[
\bm{D}_T^2:=\begin{bmatrix}
D_{1, T}^2 & \delta_{12}(T) &\cdots & \delta_{1n}(T)
\\
\delta_{21}(T) & D_{2, T}^2&\cdots & \delta_{2n}(T)
\\
\vdots&\vdots&\ddots&\vdots\\
\delta_{n1}(T)  & \delta_{n2}(T) &\cdots&D_{n, T}^2
\end{bmatrix}:\clh^n\raro\clh^n.
\]
\end{defn}
	
We establish a new notation: $\tilde{h}$ represents the general elements of $\clh^n$. In other words, we write
\[
\tilde{h} = \begin{bmatrix}h_1 & \ldots & h_n \end{bmatrix}^t,
\]
for some $h_i \in \clh$ for $i \in I_n$. Similarly, given $X_i \in \clb(\clh, \clk)$, $i \in I_n$, we have $\begin{pmatrix}X_i^*X_j\end{pmatrix}_{n\times n} : \clh^n \raro \clh^n$. It also follows that
\begin{equation}\label{norm_equality}
\langle \begin{pmatrix}X_i^*X_j\end{pmatrix}_{n\times n} \tilde{h}, \tilde{h} \rangle = \Big\|\sum_{i=1}^{n} X_i h_i\Big\|^2,
\end{equation}
for all $\tilde{h} \in \clh^n$. We now infer that for Beurling tuples, the corresponding joint defect operators are indeed positive.
	
\begin{prop}\label{positive_defect}
$\bm{D}_T^2 \geq 0$ for all $T \in \mathbb{S}_n^{B}(\clh)$.
\end{prop}
	
\begin{proof}
We again replace $T$ by $C$ on some Beurling quotient module $\clq \subseteq H^2_{\cle_*}(\D^n)$. Since the proof does not directly use the Beurling inner function, it will remain hidden. By \eqref{eqn: DjCP}, we know
\[
D_{i,C}^2=P_{\clq}M_{z_i}^*P_{\clw}M_{z_i}|_{\clq}.
\]
for all $i \in I_n$. Also, from \eqref{crosscommmmmmmmm}, we get
\[
\delta_{ij}(C) = P_{\clq}M_{z_i}^*P_{\clw}M_{z_j}|_{\clq},
\]
for $i,j \in I_n$ and $i \neq j$. This shows that
\[
\bm{D}_C^2=\begin{pmatrix}X_i^*X_j\end{pmatrix}_{n\times n},
\]
where $X_j=P_{\clw}M_{z_j}|_{\clq}$ and $j \in I_n$. Finally, by \eqref{norm_equality}, we have
\begin{equation}\label{norm_bound}
\la \bm{D}_C^2\tilde{f}, \tilde{f} \ra = \|P_{\clw}M_{z_1} f_1 + \cdots + P_{\clw}M_{z_n} f_n\|^2 \geq 0,
\end{equation}
for all $\tilde{f} \in \clq^n$, equivalently, $\bm{D}_C^2 \geq 0$.	
\end{proof}

Given a Beulring tuple $T \in \mathbb{S}_n^{B}(\clh)$, we write the non-negative square root of $\bm{D}_T^2$ as
\[
\bm{D}_T := (\bm{D}_T^2)^{\frac{1}{2}}.
\]
The corresponding \textit{joint defect space} $\bm{\cld}_T$ is defined by
\[
\bm{\cld}_T := \overline{\text{ran}}\bm{D}_T.
\]
Given a quotient module $\clq \subseteq H^2_{\cle_*}(\D^n)$, recall the definition of the wandering subspace $\clw$ (see \eqref{eqn: W_P}):
\[
\clw = \bigcap_{i=1}^n (\clq^\perp \ominus z_i \clq^\perp).
\]
If we assume, in addition, that $\clq$ is a Beurling quotient module, then we already proved in Proposition \ref{wandering_defect} that
\[
\clw = \text{clos} \Big\{\sum_{j=1}^{n}P_{\clw_{\{j\}^c}}M_{z_j}q_j : q_j \in\clq, j \in I_n \Big\}.
\]
We are now ready to connect $\clw$ and $\bm{\cld}_C$.

\begin{cor}\label{unitary}
Let $\clq$ be a Beurling quotient module of $H^2_{\cle_*}(\D^n)$. Then the operator 
\[
U(P_{\clw} M_{z_1} f_1 + \cdots + P_{\clw}M_{z_n} f_n) = \bm{D}_C \tilde{f},
\]
for all $\tilde{f} \in \clq^n$, extends to a unitary operator $U: \clw\raro \bm{\cld}_C$. In particular, $\dim \clw=\dim \bm{\cld}_C$.
\end{cor}
\begin{proof}
The result immediately follows from the above representation of $\clw$ and \eqref{norm_bound}.
\end{proof}

We make a general observation concerning Beurling tuples. Let $T\in\mathbb{S}_n^{B}(\clh)$. By Proposition \ref{corss commutator}, we know that $\overline{ran} \delta_{ij}(T) \subseteq \cld_{j,T}$ for all $i \neq j$. Moreover, Corollary \ref{product_0} implies that $\cld_{i,T} \cld_{j,T} = 0$ for all $i \neq j$. This implies
\[
\bm{\cld}_T \subseteq \bigoplus_{j \in I_n} \cld_{j,T}.
\]
Owing to the orthogonal property, we further infer that $(\bigoplus_{j \in I_n} \cld_{j,T})\subseteq \clh$. This yields the curious containment property:
\begin{equation}\label{eqn: D D H}
\bm{\cld}_T \subseteq \bigoplus_{j \in I_n} \cld_{j,T} \subseteq \clh.
\end{equation}

In other words, we can treat the joint defect space $\bm{\cld}_T$ as a closed subspace of $\clh$.

\section{Characteristic functions}\label{sec: ch fn}

After completing all the work for the construction of joint defect operators and spaces of Beurling tuples, we are now prepared for characteristic functions. Let $T \in \mathbb{S}_n^{B}(\clh)$. Recall from Theorem \ref{Dilation Thm} that the canonical dilation map $\Pi_T :\clh \raro H^2_{{\cld}_{T^*}}(\D^n)$ is an isometry and $\Pi_T T_j^* = M_{z_j}^*\Pi_T$ for all $j \in I_n$. For each $j \in I_n$, we define $\Delta_{M_{z_j},T_j} \in \clb(\clb(\clh, H^2_{{\cld}_{T^*}}(\D^n)))$ by
\[
\Delta_{M_{z_j},T_j} (X)=X-M_{z_j}XT_j^*,
\]
for all $X \in \clb(\clh, H^2_{{\cld}_{T^*}}(\D^n))$. Our first attempt at understanding the notion of characteristic functions involves a formulation of inner functions built from Beurling tuples. A priori, the representation of $\Theta_T$ in the following theorem remains less appealing due to the appearance of the canonical dilation map. However, we will make it more transparent and concrete in Theorem \ref{char fn}. 

\begin{thm}\label{main theorem}
Let $T \in \mathbb{S}_n^{B}(\clh)$. Then $\Theta_T:\D^n\raro  \clb(\bm{\cld}_T, {\cld}_{T^*})$ defined by
\[
(\Theta_T({w})) \bm{D}_T \tilde{h} = \sum_{i=1}^{n}\biggl(\Big[\prod_{j \in \{i\}^c} \Delta_{M_{z_j},T_j}(M_{z_i}\Pi_T-\Pi_TT_i)\Big] (h_i) \biggl)(w),
\]
for all $w \in \D^n$ and $\tilde{h} \in \clh^n$, is an inner function.
\end{thm}	
\begin{proof}
Since $T \in \mathbb{S}_n^{B}(\clh)$, by Theorem \ref{Beurling_qm}, $\clq:=\Pi_T\clh$ is a Beurling quotient module, and thus, there exist a Hilbert space $\cle$ and inner function $\Theta \in H^\infty_{\clb(\cle, \cld_{T^*})}(\D^n)$ such that
\begin{equation}\label{eqn: clq_theta}
\clq = \clq_\Theta.
\end{equation}
Our first goal is to compute the coefficient Hilbert space $\cle$ and connect it with our defect space of the second kind and wandering space constructed earlier. By part (1) of Lemma \ref{wandering}, we know that $\clw = \Theta \cle$; and by Proposition \ref{wandering_defect}, we also have
\[
\clw=\text{clos}\{\sum_{i=1}^{n} P_{\clw}M_{z_i}\Pi_T h_i :  h_i\in \clh, i \in I_n\}.
\]
In order to simplify $\clw$, for each $j \in I_n$, we pick $h_j\in \clh$. Now, $P_\clw M_{z_j}|_\clq = P_{\clw_{\{j\}^c}} M_{z_j}|_\clq$ by Lemma \ref{W=Wj}. Since $\Pi_T h_j \in \clq$, it follows that
\[
P_{\clw}M_{z_j}\Pi_T h_j = P_{\clw_{\{j\}^c}} M_{z_j}\Pi_T h_j.
\]
Recall that $\clw_{\{j\}^c} = \cap_{i \in \{j\}^c}(\Theta H^2_{\cle}(\D^n) \ominus z_i \Theta H^2_{\cle}(\D^n))$, and then
\begin{equation}\label{eqn: P_Wj}
P_{\clw_{\{j\}^c}} = \prod_{i \in \{j\}^c} M_{\Theta}(I-M_{z_i}M_{z_i}^*) M_{\Theta}^*,
\end{equation}
and hence
\[
P_{\clw} M_{z_j}\Pi_T h_j = M_{\Theta} \prod_{i \in \{j\}^c} (I-M_{z_i}M_{z_i}^*) M_{\Theta}^*M_{z_j}\Pi_T h_j.
\]
Summing over the $j$ then yields
\begin{equation}\label{eqn: sum of PWMPi}
\sum_{j=1}^{n} P_{\clw}M_{z_j}\Pi_T h_j = M_{\Theta} \Big[\sum_{j=1}^{n} \prod_{i \in \{j\}^c} (I-M_{z_i}M_{z_i}^*) M_{\Theta}^*M_{z_j}\Pi_T h_j \Big].
\end{equation}
As $\cle = M_\Theta^* \clw$, we must therefore have
\[
\cle=\text{clos} \left\{\sum_{i=1}^{n} \prod_{j \in \{i\}^c} (I-M_{z_j}M_{z_j}^*)M_{\Theta}^*M_{z_i}\Pi_T h_i : h_i\in \clh \right\},
\]
the needed representation of the coefficient space $\cle$. Now we connect this with the joint defect spaces as follows: Define $U:\cle\raro \bm{\cld}_T$ by
\[
U\left(\sum_{i=1}^{n} \prod_{j \in \{i\}^c}(I-M_{z_j}M_{z_j}^*)M_{\Theta}^*M_{z_i}\Pi_T h_i \right) = \bm{D}_T \tilde{h},
\]
for all $\tilde{h} \in \clh^n$. For each $\tilde{h} \in \clh^n$, by \eqref{eqn: sum of PWMPi}, we have
\[
\begin{split}
\Biggl\lvert\Biggl\lvert \left(\sum_{i=1}^{n} \prod_{j \in \{i\}^c}(I-M_{z_j}M_{z_j}^*)M_{\Theta}^*M_{z_i}\Pi_T h_i \right)\Biggl\rvert\Biggl\rvert&=
\Biggl\lvert\Biggl\lvert	\left(   \sum_{i=1}^{n} \prod_{j \in \{i\}^c}M_{\Theta}(I-M_{z_j}M_{z_j}^*)M_{\Theta}^*M_{z_i}\Pi_T h_i    \right)\Biggl\rvert\Biggl\rvert
\\
&=\Biggl\lvert\Biggl\lvert \sum_{j=1}^{n} P_{\clw_{\{j\}^c}} M_{z_j}\Pi_T h_j   \Biggl\rvert\Biggl\rvert
\\
&= \| \bm{D}_T \tilde{h} \|,
\end{split}
\]
where the last identity follows from Corollary \ref{unitary}. Therefore, $U$ extends to a unitary operator. Define
\begin{equation}\label{eqn: theta U}
\Theta_T(z):=\Theta(z)U^* \qquad (z \in \D^n).
\end{equation}
Clearly, $\Theta_T:\D^n \raro \clb(\bm{\cld}_T, \cld_{T^*})$ is an inner function and $M_{\Theta_{T}} \bm{D}_T \tilde{h} = M_\Theta U^* \bm{D}_T \tilde{h}$, $\tilde{h} \in \clh^n$. By the definition of $U$, this implies
\[
M_{\Theta_{T}} \bm{D}_T \tilde{h} =\sum_{i=1}^{n} \prod_{j \in \{i\}^c}M_{\Theta}(I-M_{z_j}M_{z_j}^*)M_{\Theta}^*M_{z_i}\Pi_T h_i =\sum_{i=1}^{n} P_{\clw_{\{i\}^c}} M_{z_i}\Pi_T h_i,
\]
where the final identity is due to \eqref{eqn: sum of PWMPi}. Noting $\Delta_{M_{z_t}} P_{\cls} = M_{z_t} P_\cls M_{z_t}^*$, $t \in I_n$, for each $i \neq j$ and $h_i \in \clh$, we compute
\[
((I-\Delta_{M_{z_j}})P_{\cls}) M_{z_i}\Pi_T h_i = (P_{\cls}- M_{z_j}P_{\cls} M_{z_j}^*)M_{z_i}\Pi_T h_i =P_{\cls}M_{z_i}\Pi_T h_i- M_{z_j}P_{\cls} M_{z_j}^* M_{z_i}\Pi_T h_i.
\]
We must simplify both terms of the final identity. We recall the dilation property that $\Pi_T^* M_{z_i} = T_i \Pi_T$ for all $i \in I_n$. In conjunction with $P_\cls = I - \Pi_T \Pi_T^*$, we streamline the first term as follows:
\[
P_{\cls}M_{z_i}\Pi_T h_i =(I-\Pi_T\Pi_T^*)M_{z_i}\Pi_T h_i =M_{z_i}\Pi_T h_i-  \Pi_T\Pi_T^* M_{z_i}\Pi_T h_i =M_{z_i}\Pi_T h_i- \Pi_TT_i h_i,
\]
where the second term simplifies as
\[
\begin{split}
M_{z_j}P_{\cls} M_{z_j}^* M_{z_i}\Pi_T h_i & = M_{z_j}(I-\Pi_T\Pi_T^*) M_{z_i}  M_{z_j}^* \Pi_T h_i
\\
&=M_{z_j}(I-\Pi_T\Pi_T^*) M_{z_i}  \Pi_T T_j^* h_i
\\
&=M_{z_j}M_{z_i}  \Pi_T T_j^* h_i-M_{z_j}\Pi_T(\Pi_T^* M_{z_i} \Pi_T) T_j^* h_i
\\
&=M_{z_j}M_{z_i}  \Pi_T T_j^* h_i-M_{z_j}\Pi_T T_iT_j^* h_i
\\
&=M_{z_j}(M_{z_i}\Pi_T - \Pi_TT_i ) T_j^* h_i.
\end{split}
\]
Summarizing the above two sets of simplification, we have
\[
\begin{split}
[( I-\Delta_{M_{z_j}}) P_{\cls}] M_{z_i}\Pi_T h_i &=(M_{z_i}\Pi_T -  \Pi_TT_i ) h_i -M_{z_j}(M_{z_i}\Pi_T - \Pi_TT_i ) T_j^* h_i
\\
&=\left[\Delta_{M_{z_j},T_j}(M_{z_i}\Pi_T -  \Pi_TT_i )\right] h_i.
\end{split}
\]
We can rewrite \eqref{eqn: P_Wj} as (note that $P_\cls = M_{\Theta}M_{\Theta}^*$)
\[
P_{\clw_{\{i\}^c}} = \prod_{j \in \{i\}^c}\left( I-\Delta_{M_{z_j}} \right)(P_{\cls}),
\]
for all $i \in I_n$, and hence, in order to simplify $M_{\Theta_{T}} \bm{D}_T \tilde{h}$, we compute
\[
P_{\clw_{\{i\}^c}} M_{z_i}\Pi_T h_i = \left[\prod_{j \in \{i\}^c}\left( I-\Delta_{M_{z_j}} \right)(P_{\cls})\right]M_{z_i}\Pi_T h_i = \left[\prod_{j \in \{i\}^c} \Delta_{M_{z_j},T_j}(M_{z_i}\Pi_T -  \Pi_TT_i )\right] h_i,
\]
for all $i \in I_n$ and $h_i \in \clh$, and consequently
\[
M_{\Theta_{T}} \bm{D}_T \tilde{h} = \sum_{i=1}^{n} P_{\clw_{\{i\}^c}} M_{z_i}\Pi_T h_i = \sum_{i=1}^{n} \left[\prod_{j \in \{i\}^c} \Delta_{M_{z_j},T_j}(M_{z_i}\Pi_T -  \Pi_TT_i )\right] h_i,
\]
which completes the proof of the theorem.
\end{proof}

The function $\Theta_T$ defined in Theorem \ref{main theorem} is our candidate for the characteristic function of the tuple $T$. However, we cannot explicitly define the function at this time because it involves the canonical isometric dilation map. The following theorem eliminates this drawback. It also yields concrete analytic models for Beurling tuples. 

\begin{thm}\label{char fn}
Let $T \in \mathbb{S}_n^{B}(\clh)$. Then the operator-valued analytic function $\Theta_{T}$ presented in Theorem \ref{main theorem} is given by
\[
\Theta_{T}(w)\bm{D}_T \tilde{h} = D_{T^*}\prod_{k=1}^{n}(I_{\clh}-w_kT_k^*)^{-1}\sum_{j=1}^{n}(w_jI_{\clh}-T_j)\prod_{i \in \{j\}^c} (I_{\clh}-w_iT_i^*)h_j,
\]
for all $w \in \D^n$ and $\tilde{h} \in \clh^n$. Moreover, $T \cong C_{\clq_{\Theta_T}}$, where $C_{\clq_{\Theta_T}}$ is the tuple of model operators defined on
\[
\clq_{\Theta_{T}}=H^2_{\cld_{T^*}}(\D^n)\ominus \Theta_T H^2_{\bm{\cld}_T}(\D^n).
\]
\end{thm}
	
\begin{proof}
In order to shorten notation, for each $i \in I_n$, we define
\[
\Pi_{z_i, T}:= M_{z_i} \Pi_T - \Pi_T M_{z_i}.
\]
The expression of $\Theta_T$ in Theorem \ref{main theorem} then becomes
\begin{equation}\label{eqn: MThetaTDT}
M_{\Theta_{T}} \bm{D}_T \tilde{h} = \sum_{i=1}^{n} \left[\prod_{j \in \{i\}^c}\Delta_{M_{z_j},T_j}\Pi_{z_i, T} \right] h_i,
\end{equation}
for all $\tilde{h} \in \clh^n$. For each $i\in I_n$, $\tilde{h} \in \clh^n$, and $w\in \D^n$, we get
\[
\begin{split}
(\Pi_{z_i, T} h_i)(w )& = w_i {D}_{T^*} \prod_{k=1}^{n}(I_{\clh}-w_k T_k^*)^{-1} h_i
- {D}_{T^*}\prod_{k=1}^{n}(I_{\clh}-w_k T_k^*)^{-1} T_ih_i
\\
& = {D}_{T^*}\prod_{k=1}^{n}(I_{\clh}-w_k T_k^*)^{-1}\left( w_iI_{\clh}-T_i \right)h_i.
\end{split}
\]
For each $i \neq j$, we have $\Delta_{M_{z_j},T_j} \Pi_{z_i, T} = \Pi_{z_i, T} - M_{z_j} \Pi_{z_i, T} T_j^*$, and hence
\[
\left(\Delta_{M_{z_j},T_j} \Pi_{z_i, T} h_i \right)(w) = D_{T^*}\prod_{k=1}^{n}(I_{\clh}-w_k T_k^*)^{-1}\left( w_iI_{\clh}-T_i \right)(I_{\clh}-w_jT_j^*)h_i,
\]
which implies
\[
\left[\left(\prod_{j \in \{i\}^c} \Delta_{M_{z_j},T_j} \Pi_{z_i, T} \right)(h_i)\right](w) = {D}_{T^*}\prod_{k=1}^{n}(I-w_k T_k^*)^{-1}\left( w_iI-T_i \right)  \prod_{j \in \{i\}^c} (I-w_jT_j^*)h_i.
\]
Therefore
\[
\sum_{i=1}^{n} \Big[\prod_{j \in \{i\}^c}\Delta_{M_{z_j},T_j} \Pi_{z_i, T}\Big] h_i = {D}_{T^*}\prod_{k=1}^{n}(I_{\clh}-w_k T_k^*)^{-1} \sum_{i=1}^n \left( w_i I_{\clh}-T_i \right)  \prod_{j \in \{i\}^c} (I_{\clh}-w_jT_j^*)h_i.
\]
We now arrive at the result by comparing this with the identity \eqref{eqn: MThetaTDT}. For the final part, we recall from \eqref{eqn: theta U} that $\Theta_T(z) =\Theta(z)U^*$ for all $z \in \D^n$. In other words, $M_{\Theta_T} = M_\Theta M_U$, where $M_U$ denotes the constant unitary operator. Then
\[
M_{\Theta_T}M_{\Theta_T}^* = M_{\Theta} M_U^* M_U M_{\Theta}^*=M_{\Theta}M_{\Theta}^* = P_{\Theta H^2_{\cle}(\D^n)},
\]
where $\Theta$ is as in \eqref{eqn: clq_theta} in the proof of Theorem \ref{main theorem}. Also recall from the proof of Theorem \ref{main theorem} that $\clq = \Pi_T \clh$ and
\[
\clq = \clq_\Theta = (\Theta H^2_{\cle}(\D^n))^\perp,
\]
and consequently, $\Pi_T \Pi_T^* = I - M_{\Theta_T}M_{\Theta_T}^*$. This proves that $\Pi_T\clh = \clq = \clq_{\Theta_{T}}$. Finally, $\Pi_T: \clh \raro \Pi_T \clh = \clq_{\Theta_{T}}$ is the required unitary satisfying the intertwining property that
\[
\Pi_T T_i^* = (P_{\clq_{\Theta_T}} M_{z_i}^*|_{\clq_{\Theta_T}}) \Pi_T,
\]
for all $i \in I_n$. This completes the proof of the theorem.
\end{proof}

We are now in a position to define the inner characteristic functions attached to Beurling tuples:
 
\begin{defn}\label{def: char funct}
Let $T \in \mathbb{S}_n^B(\clh)$. The characteristic function of $T$ is the operator-valued analytic function $\Theta_T : \D^n \raro \clb(\bm{\cld}_T, \cld_{T^*})$ defined by
\[
\Theta_{T}(w)\bm{D}_T \tilde{h} = D_{T^*}\prod_{k=1}^{n}(I_{\clh}-w_kT_k^*)^{-1}\sum_{j=1}^{n}(w_jI_{\clh}-T_j)\prod_{i \in \{j\}^c} (I_{\clh}-w_iT_i^*)h_j,
\]
for all $w \in \D^n$ and $\tilde{h} \in \clh^n$.
\end{defn}

We remind the reader once again that $T \in \mathbb{S}_n^B(\clh)$ is a necessary condition for the tuple $T$ to admit a characteristic function.

It is now useful to compare the above definition for $n > 1$ with the classical Sz.-Nagy and Foias's characteristic functions in the format outlined at the end of Section \ref{sec: basics}. Of course, the occurrences of additional terms and factors that appeared in the above due to the effect of $n > 1$.

One may naturally question the effectiveness of our characteristic functions. Here, we aim to prove that these functions serve as a complete unitary invariant, aligning with the classical case. In particular, this result establishes that the construction of the characteristic function thus far is a canonical choice, extending to representations of inner Schur functions on the polydisc. First, we introduce the notion of equality at the level of operator-valued analytic functions.
	
\begin{defn}\label{def: coinc}
Let $T=(T_1,\dots,T_n) \in \mathbb{S}_n^B(\clh)$ and $S=(S_1,\dots,S_n) \in \mathbb{S}_n^B(\clk)$ be two Beurling tuples. We say that the characteristic functions $\Theta_T$ and $\Theta_S$ coincide if there exist unitaries $\tau:\bm{\cld}_S\raro\bm{\cld}_{T}$ and $\tau_*: \cld_{S^*}\raro \cld_{T^*}$ such that
\[
\Theta_T(w) = \tau_* \Theta_S(w) \tau^*,
\]
for all $w\in \D^n$. In short, we write this as $\Theta_T \cong \Theta_S$.
\end{defn}

We prove that the notions of joint unitarity equivalence and coinciding characteristic functions are just the same. 

\begin{thm}\label{thm: compl unit inv}
Let $T \in \mathbb{S}_n^B(\clh)$ and let $S \in \mathbb{S}_n^{B}(\clk)$. Then $T\cong S$ if and only if $\Theta_T \cong \Theta_S$.
\end{thm}
\begin{proof}
Assume that $\Theta_T \cong \Theta_S$, that is, there are unitary operators $\tau:\bm{\cld}_S\raro\bm{\cld}_{T}$ and $\tau_*: {\cld}_{S^*}\raro {\cld}_{T^*}$ such that $\tau_*^* \Theta_T(w)\tau=\Theta_S(w)$ for all $w\in \D^n$. Recall that $\mathbb{S}_n$ denotes the Szeg\"{o} kernel of $\D^n$. For each $w \in \D^n$, we define the kernel function $\mathbb{S}_n(\cdot, w)$ by
\[
(\mathbb{S}_n(\cdot, w))(z) = \mathbb{S}_n(z, w) \qquad (z \in \D^n).
\]
Let $w\in \D^n$ and $\eta\in {\cld}_{S^*}$. We use the standard reproducing kernel Hilbert space property: Since $\Theta_T$ is a Schur function (or a multiplier), it follows that
\[
M_{\Theta_T}^* (\mathbb{S}_n(\cdot, w) \otimes \eta) = \mathbb{S}_n(\cdot, w) \otimes \Theta_T(w)^* \eta,
\]
for all $w \in \D^n$ and $\eta \in {\cld}_{S^*}$. Note that $M_{\Theta_S}^*$ also shares a similar identity. For each $w \in \D^n$ and $\eta \in \cld_{S^*}$, we compute
\[
\begin{split}
(I\otimes \tau^*)M_{\Theta_T}^*(I\otimes \tau_*)(\mathbb{S}_n(\cdot, w) \otimes \eta) & = (I\otimes \tau^*)M_{\Theta_T}^*(\mathbb{S}_n(\cdot, w) \otimes \tau_*\eta)
\\
&=(I\otimes \tau^*)(\mathbb{S}_n(\cdot, w) \otimes\Theta_T(w)^*\tau_*\eta)
\\
& = \mathbb{S}_n(\cdot, w)\otimes(\tau^* \Theta_T(w)^*\tau_*\eta)
\\
&= \mathbb{S}_n(\cdot, w) \otimes \Theta_T(w)^*\eta
\\
&=M_{\Theta_S}^*(\mathbb{S}_n(\cdot, w) \otimes \eta).
\end{split}
\]
Thus $(I\otimes \tau^*)M_{\Theta_T}^*(I\otimes \tau_*)=M_{\Theta_S}^*$, which implies
\[
(I\otimes \tau^*_*)	M_{\Theta_T}M_{\Theta_T}(I\otimes \tau_*)=M_{\Theta_S}M_{\Theta_S}^*.
\]
In particular,
\[
P_{\clq_{\Theta_T}}(I\otimes \tau_*)=(I\otimes \tau_*)P_{\clq_{\Theta_S}}.
\]
Moreover, for each $j \in I_n$, we have
\[
\begin{split}
(I\otimes \tau_*)P_{\clq_{\Theta_{S}}}M_{z_j}P_{\clq_{\Theta_{S}}}
&=P_{\clq_{\Theta_T}}(I\otimes \tau_*)M_{z_j}P_{\clq_{\Theta_{S}}}
\\
&=P_{\clq_{\Theta_T}} M_{z_j} (I\otimes \tau_*)P_{\clq_{\Theta_{S}}}
\\
&= P_{\clq_{\Theta_T}} M_{z_j}  P_{\clq_{\Theta_T}}(I\otimes \tau_*).
\end{split}
\]
This shows that
\[
(P_{\clq_{\Theta_{S}}}M_{z_1}|_{\clq_{\Theta_{S}}},\dots, P_{\clq_{\Theta_{S}}}M_{z_n}|_{\clq_{\Theta_{S}}}) \cong (P_{\clq_{\Theta_T}} M_{z_1}|_{\clq_{\Theta_T}},\dots, P_{\clq_{\Theta_T}} M_{z_n}  |_{\clq_{\Theta_T}}),
\]
equivalently, $T \cong S$. For the converse, suppose that there is a unitary $\sigma \in \clb(\clh, \clk)$ such that
\[
\sigma^*S_j\sigma =T_j,
\]
for all $j \in I_n$. Then
\[
\sigma^* {D}_{S^*}\sigma = {D}_{T^*},
\]
and consequently, $\tau_*: {\cld}_{T^*}\raro {\cld}_{S^*}$ is a unitary map, where
\[
\tau_* ({D}_{T^*}h) = {D}_{S^*}\sigma h \qquad (h\in \clh).
\]
Also, we have the unitary operator $\Sigma = \sigma \oplus \cdots \oplus \sigma \in \clb(\clh^n, \clk^n)$ satisfying
\[
\Sigma^* {D}_{T}\Sigma = {D}_{S}.
\]
Similarly, we define the other unitary operator $\tau: \bm{\cld}_T\raro \bm{\cld}_S$ by
\[
\tau \left(\bm{D}_{T}\tilde{h} \right) = \bm{D}_{S}\Sigma \tilde{h} \qquad (\tilde{h} \in \clh^n).
\]
For each $\tilde{h} \in \clh^n$ and $w \in \D^n$, we finally compute
\[
\begin{split}
\tau^*_*\Theta_{S}(w)\tau \bm{D}_T\tilde{h}&= \tau_*^*\Theta_{S}(w)\bm{D}_{S}\Sigma \tilde{h}
\\
&=\tau^*_* D_{S^*}\prod_{k=1}^{n}(I_{\clh}-w_kS_k^*)^{-1}\sum_{j=1}^{n}(w_jI_{\clh}-S_j)\prod_{i \in \{j\}^c} (I_{\clh}-w_iS_i^*)\sigma h_j
\\
&=\tau^*_* D_{S^*}\sigma\prod_{k=1}^{n}(I_{\clh}-w_kT_k^*)^{-1}\sum_{j=1}^{n}(w_jI_{\clh}-T_j)\prod_{i \in \{j\}^c} (I_{\clh}-w_iT_i^*)h_j
\\
&=D_{T^*}\prod_{k=1}^{n}(I_{\clh}-w_kT_k^*)^{-1}\sum_{j=1}^{n}(w_jI_{\clh}-T_j)\prod_{i \in \{j\}^c} (I_{\clh}-w_iT_i^*) h_j
\\
&=\Theta_T(w) \bm{D}_T\tilde{h},
\end{split}
\]
which implies $\tau_*^*\Theta_{S}(w)\tau=\Theta_T(w)$, thereby completing the proof of the theorem.
\end{proof}

When $n=1$, this recovers the classical Sz.-Nagy and Foias's result of contractions at the level of $C_{\cdot 0}$ contractions.

\section{Concluding remarks}\label{sec: classical}

In this concluding section, we bring out results that are both relevant to our context and arise as applications of the observations obtained thus far. We hope that some of these results will prove useful in addressing other problems, such as commutators for commuting tuples of operators. We begin by comparing the characteristic functions derived in this paper with the classical ones.

\subsection{On classical models}\label{sub sect: classical models} This subsection aims to show how our results on the $n$-polydisc align with the classical result of Sz-Nagy and Foias, assuming $n=1$. Suppose $T$ is a pure contraction on $\clh$ (that is, $T \in C_{\cdot 0}$). In this case, the canonical dilation map $\Pi_T:\clh\raro H^2_{\cld_{T^*}}(\D)$ is given by (see Theorem \ref{Dilation Thm})
\[
(\Pi_T h)(z)=D_{T^*}(I_{\clh}-zT^*)^{-1}h,
\]
for all $h\in \clh$ and $z\in \D$. Moreover, $\Pi_T$ is an isometry, and $\Pi_T T^*=M_z^*\Pi_T$. Set $\clq:=\Pi_T\clh$. As $\clq$ is a quotient module, there exist a Hilbert space $\cle$ and an inner function $\Theta \in H^{\infty}_{\clb(\cle,\cld_{T^*})}(\D)$ such that
\[
\cls:=\clq^{\perp} =\Theta H^2_{\cle}(\D).
\]
We know that $\cls\ominus z\cls=\Theta \cle$. Moreover, as in the proof of Proposition \ref{wandering_defect}, one can prove that $ \text{clos}(P_{\cls}M_z\clq) = \cls\ominus z\cls$. Therefore,
\[
\cle = \text{clos}(M_{\Theta}^*M_z\Pi_T).
\]
Since $P_\cls = I - \Pi_T \Pi_T^*$, it follows that
\[
\begin{split}
P_{\cls}M_z\Pi_T = M_z\Pi_Th-\Pi_T \Pi_T^* M_z\Pi_T = M_z\Pi_T-\Pi_TT.
\end{split}
\]
Now, for $h \in \clh$, we compute
\[
\begin{split}
\|M_{\Theta}^*M_z\Pi_Th\|^2&=\la M_{\Theta}^*M_z\Pi_Th, M_{\Theta}^*M_z\Pi_Th \ra
\\
&=\la \Pi_T^* M_z^*M_{\Theta}M_{\Theta}^*M_z\Pi_Th,\; h  \ra
\\
&=\la h,h \ra -\la \Pi_T^* M_z^*P_{\clq}M_z\Pi_Th,\; h \ra
\\
&=\la h,h \ra -\la \Pi_T^* M_z^*\Pi_T\Pi_T^*M_z\Pi_Th,\; h \ra
\\
&=\la h,h \ra -\la T^*Th,h\ra
\\
&=\|D_T h\|^2.
\end{split}
\]
We define $U:  \text{clos}(M_{\Theta}^*M_z\Pi_T)\raro \cld_T$ by
\[
U(M_{\Theta}^*M_z\Pi_Th)=D_Th,
\]
for all $h\in \clh$. So, $U$ is a unitary operator. Hence, $U^*D_T h=M_{\Theta}^*M_z\Pi_Th$. We define $\Theta_T(\cdot):=\Theta(\cdot) U^*$ (recall \eqref{eqn: theta U}). Then $\Theta_T:\D\raro\clb(\cld_T,\cld_{T^*})$ is an inner function and
\[
\begin{split}
M_{\Theta_T} D_Th = M_{\Theta} U^*D_Th = M_{\Theta}M_{\Theta}^*M_z\Pi_Th = P_{\cls}M_z\Pi_Th = (M_z\Pi_T-\Pi_TT)(h).
\end{split}
 \]
Thus for $w\in \D$, we have
\[
\Theta_{T}(w)D_Th=[(M_z\Pi_T-\Pi_TT)(h)](w).
\]
We are now precisely at the stage of Theorem \ref{main theorem} and proceed as in the setting of Theorem \ref{char fn}. In other words, to find the expression of $\Theta_T$ we compute:
\[
\begin{split}
\Theta_{T}(w)D_Th
&=
[(M_z\Pi_T-\Pi_TT)(h)](w)
\\
&=w (\Pi_Th)(w)-(\Pi_TTh)(w)
\\
&=wD_{T^*}(I-wT^*)^{-1}h-D_{T^*}(I-wT^*)^{-1}Th
\\
&=wD_{T^*}(I-wT^*)^{-1}h-D_{T^*}Th-wD_{T^*}(I-wT^*)^{-1}T^*Th
\\
&=-D_{T^*}Th+wD_{T^*}(I-wT^*)^{-1}(I-T^*T)h
\\
&=-TD_Th+wD_{T^*}(I-wT^*)^{-1}D_T^2h
\\
&=[-T+wD_{T^*}(I-wT^*)^{-1}D_T]D_Th.
\end{split}
\]
As $D_T \clh$ is dense in $\cld_T$, it finally follows that
\[
\Theta_{T}(w) = [-T+wD_{T^*}(I-wT^*)^{-1}D_T]|_{\cld_T} \qquad (w \in \D).
\]
Hence $\Theta_T$ matches with the Sz-Nagy and Foias characteristic function of $T$ (see \eqref{eqn: ch fn 1 var}).

\subsection{Representing inner functions}\label{sub sec: inner funct} In the introduction, we highlighted the potential of the characteristic function formula to yield representations of inner functions on $\D^n$, $n \geq 1$. Here, we present the theory and provide a detailed description of inner functions. We fix a nonconstant inner function $\Theta \in H^\infty_{\clb(\cle,\cle_*)}(\D^n)$, and set
\[
\clr = \bigvee_{k \in \Z_+^n} z^k \clq_{\Theta}.
\]
Clearly, by construction, $\clr$ is a submodule of $H^2_{\cle_*}(\D^n)$. On the other hand, since $M_{z_i}$ is an isometry, it follows easily that $M_{z_i}^* \clr \subseteq \clr$ for all $i \in I_n$. In other words, $\clr$ is both a submodule and a quotient module, or equivalently, $\clr$ reduces the tuple $(M_{z_1}, \ldots, M_{z_n})$. There exist closed subspaces $\cle_*^{\prime}$ and $\cle_*^{\prime\prime}$ such that
\[
\cle_*=\cle_*^{\prime}\oplus \cle_*^{\prime\prime},
\]
and $\clr = H^2_{\cle_*^{\prime}}(\D^n)$. It follows that
\[
H^2_{\cle_*}(\D^n) = \clq_{\Theta}\oplus \Theta H^2_{\cle}(\D^n) = H^2_{\cle_*^{\prime}}(\D^n)\oplus H^2_{\cle_*^{\prime\prime}}(\D^n).
\]
Since $\clq_{\Theta}\subseteq H^2_{\cle_*^{\prime}}(\D^n)$ and $\Theta H^2_{\cle}(\D^n) = H^2_{\cle}(\D^n) \ominus \clq_{\Theta}$, we get
\[
\Theta H^2_{\cle}(\D^n)=(H^2_{\cle_*^{\prime}}(\D^n)\ominus\clq_{\Theta})\oplus H^2_{\cle_*^{\prime\prime}}(\D^n).
\]
This shows that $M_{\Theta}$ maps $H^2_{\cle}(\D^n)$ onto $(H^2_{\cle_*^{\prime}}(\D^n)\ominus\clq_{\Theta})\oplus H^2_{\cle_*^{\prime\prime}}(\D^n)$.
Define closed subspaces of $H^2_{\cle}(\D^n)$ by
\[
\clm:=\{f\in H^2_{\cle}(\D^n): \Theta f\in H^2_{\cle_*^{\prime\prime}}(\D^n) \},
\]
and
\[
\cln = \{f\in H^2_{\cle}(\D^n): \Theta f\in  H^2_{\cle_*^{\prime}}(\D^n)\ominus\clq_{\Theta}\},
\]
so that $\clm\oplus \cln=H^2_{\cle}(\D^n)$. By the construction of $\clm$, if $f\in \clm$, then it readily follows that $z_jf\in \clm$ for all $j \in I_n$, and hence $\clm$ is a submodule of $H^2_{\cle}(\D^n)$. We also observe that $\clm=M_{\Theta}^*H^2_{\cle_*^{\prime\prime}}(\D^n)$, that is, $\clm$ is a quotient module. This implies that $\clm$ reduces the tuple $(M_{z_1}, \ldots, M_{z_n})$ on $H^2_{\cle}(\D^n)$. So there are closed subspaces $\cle^{\prime}$ and $\cle^{\prime\prime}$ such that
\[
\cle=\cle^{\prime}\oplus \cle^{\prime\prime}.
\]
and
\[
\clm=H^2_{\cle^{\prime\prime}}(\D^n).
\]
Then $\Theta H^2_{\cle^{\prime}}(\D^n) = H^2_{\cle_*^{\prime}}(\D^n)\ominus\clq_{\Theta}$ and $\Theta H^2_{\cle^{\prime\prime}}(\D^n)=H^2_{\cle_*^{\prime\prime}}(\D^n)$, and consequently, there exist a unitary $W:\cle^{\prime\prime}\raro \cle_*^{\prime\prime}$ and inner function $\Theta^{\prime} \in H^\infty_{\clb(\cle^{\prime}, \cle^{\prime}_*)}(\D^n)$ such that
\[
M_{\Theta}=\begin{bmatrix}
M_{\Theta^{\prime}}&0\\0&M_W
\end{bmatrix}:H^2_{\cle^{\prime}}(\D^n)\oplus H^2_{\cle^{\prime\prime}}(\D^n)\raro
H^2_{\cle_*^{\prime}}(\D^n)\oplus H^2_{\cle_*^{\prime\prime}}(\D^n).
\]
Clearly, this multiplication operator corresponds to the decomposition of the inner function
\[
\Theta(z)=\begin{bmatrix}
\Theta^{\prime}(z)&0\\0&W
\end{bmatrix}:\cle^{\prime}\oplus\cle^{\prime\prime} \raro \cle_*^{\prime}\oplus\cle_*^{\prime\prime},
\]
for all $z\in \D^n$. Define $T \in \mathbb{S}^B_n(\clq_\Theta)$ by
\[
T:=(P_{\clq_{\Theta}}M_{z_1}|_{\clq_{\Theta}},\dots, P_{\clq_{\Theta}}M_{z_n}|_{\clq_{\Theta}}).
\]
On one hand, by the construction, $(M_{z_1}, \ldots, M_{z_n})$ on $H^2_{\cle_*^{\prime}}(\D^n)$ is the minimal dilation of $T$, where, on the other hand, the canonical dilation $(M_{z_1}, \ldots, M_{z_n})$ on $H^2_{\cld_{T^*}}(\D^n)$ of $T$ is also a minimal one (see Theorem \ref{Dilation Thm}). This forces
\[
\dim \cle_*^{\prime}= \dim \cld_{T^*}.
\]
The uniqueness of the dilation in this setting leads us to a unitary $V:\cle_*^{\prime}\raro \cld_{T^*}$ which yields a (constant) unitary multiplier $M_{V} \in \clb(H^2_{\cle_*^{\prime}}(\D^n), H^2_{\cld_{T^*}}(\D^n))$. We have
\[
M_{V}\clq_{\Theta}=\clq_{\Theta_T},
\]
and
\[
M_{V} (H^2_{\cle_*^{\prime}}(\D^n)\ominus \clq_{\Theta}) = \Theta_T H^2_{\bm{\cld}_{T}}(\D^n),
\]
where $\Theta_T:\D^n \raro\clb(\bm{\cld}_T,\cld_{T^*})$ is the characteristic function of $T$ (see Definition \ref{def: char funct}). Since $\Theta^{\prime} H^2_{\cle^{\prime}}(\D^n) = H^2_{\cle_*^{\prime}}(\D^n)\ominus\clq_{\Theta}$, we get
\[
V\Theta^{\prime} H^2_{\cle^{\prime}}(\D^n) = \Theta_T H^2_{\bm{\cld}_{T}}(\D^n).
\]
This again leads us to a unitary $V^{\prime}: \cle^{\prime} \raro \bm{\cld}_T$ such that
\[
V\Theta^{\prime}(z) = \Theta_T(z)V^{\prime},
\]
for all $z\in \D^n$. Define unitary operators $\sigma$ and $\sigma_{*}$ by
\[
\sigma_*:=\begin{bmatrix}
V&0\\0&I_{\cle_*^{\prime\prime}}
\end{bmatrix}: \cle_*^{\prime}\oplus\cle_*^{\prime\prime}\raro \cld_{T^*}\oplus \cle_*^{\prime\prime},
\]
and
\[
\sigma:=\begin{bmatrix}
V^{\prime}&0\\0&W
\end{bmatrix}: \cle^{\prime}\oplus\cle^{\prime\prime}\raro \bm{\cld}_{T}\oplus \cle_*^{\prime\prime}.
\]
Then
\[
\sigma_*^{*}\begin{bmatrix}
\Theta_T(z)&0\\0&I_{\cle_*^{\prime\prime}}
\end{bmatrix}\sigma=
\begin{bmatrix}
V^*\Theta_T(z)V^{\prime}&0\\0&W
\end{bmatrix}
=\begin{bmatrix}
\Theta^{\prime}(z)&0\\0&W
\end{bmatrix}
=\Theta(z),
\]
for all $z\in \D^n$ implies that $\Theta$ coincides with $\Theta_T \oplus  I_{\cle_*^{\prime\prime}}$. In summary, up to (the unitary) coincidence of analytic functions (see Definition \ref{def: coinc}), we have the representation: 
\begin{equation}\label{eqn: inner fn}
\Theta(z) = \begin{bmatrix}
\Theta_T(z)&0\\0&I_{\cle_*^{\prime\prime}}
\end{bmatrix},
\end{equation}
for all $z \in \D^n$, for some $T \in \mathbb{S}_n^B(\clh)$ on some Hilbert space $\clh$, where $\Theta_T$ is the characteristic function of $T$ and is given by (see Definition \ref{def: char funct})
\[
\Theta_{T}(z)\bm{D}_T \tilde{h} = D_{T^*}\prod_{k=1}^{n}(I_{\clh}-z_k T_k^*)^{-1}\sum_{j=1}^{n}(z_jI_{\clh}-T_j)\prod_{i \in \{j\}^c} (I_{\clh}-z_iT_i^*)h_j,
\]
for all $z \in \D^n$ and $\tilde{h} \in \clh^n$. Therefore, \eqref{eqn: inner fn} is the desired representation of inner functions $\Theta$ on $\D^n$.

\subsection{Commutator defect operators}\label{sub: new comm}
In the context of the joint defect operators introduced in Definition \ref{defect}, this subsection introduces an additional operator that may provide further insight into defining the notion of joint commutators for commuting tuples of operators (however, see \cite{Athavale, Curto}). However, the results presented in this section do not have any significant implications for the overall paper. In this case, we define another operator that would eventually dominate the joint defect operator:

\begin{defn}\label{def: commutator defect}
For each $T \in \clb^n_c(\clh)$, define the commutator defect operator $\bm{D}_{c, T}^2$ by
\[
\bm{D}_{c, T}^2=\begin{bmatrix}
D_{T_1}^2 &[T_2, T_1^*] &\cdots &[T_n, T_1^*]\\
[T_1, T_2^*] & D_{T_2}^2&\cdots &[T_n, T_2^*]\\
\vdots&\vdots&\ddots&\vdots\\
[T_1, T_n^*] &[T_2, T_n^*] &\cdots&D_{T_n}^2
\end{bmatrix}.
\]
\end{defn}

And here is the dominance property:

\begin{thm}
$\bm{D}_{c, T}^2\geq 0$ for all $T\in \mathbb{S}_n(\clh)$.  Moreover, if $T\in\mathbb{S}^B_n(\clh)$, then $\bm{D}_{T}^2 \leq \bm{D}_{c, T}^2$.
\end{thm}
\begin{proof}
We consider the $n$-tuple $C$ on some quotient module $\clq \subseteq H^2_{\cle_*}(\D^n)$ instead of $T$ (in view of Theorem \ref{Dilation Thm}). By \eqref{lemma: defect Ci}, we know $D_{C_i}^2=P_{\clq}M_{z_i}^*P_{\cls}M_{z_i}|_{\clq}$ for all $i \in I_n$. Also, by Proposition \ref{corss_commutator}, we have $[C_j,C_i^*]=P_{\clq}M_{z_i}^*P_{\cls}M_{z_j}|_{\clq}$ for all $i \neq j$. This shows that
\[
\bm{D}_{c, C}^2=\begin{pmatrix} Y_i^*Y_j\end{pmatrix}_{n\times n},
\]
where $Y_j=P_{\cls}M_{z_j}|_{\clq}$ and $j \in I_n$. By \eqref{norm_equality}, we have
\begin{equation}\label{norm_bound_2}
\la \bm{D}_{c, C}^2 \tilde{f},  \tilde{f}\ra=\|P_{\cls}M_{z_1} f_1+\dots+P_{\cls}M_{z_n} f_n\|^2 \geq 0,
\end{equation}
for all $\tilde{f} \in \clq^n$, equivalently, $\bm{D}_{c, C}^2 \geq 0$, which proves the first part. Now, we assume that $\clq$ is a Beurling quotient module, that is, $C\in\mathbb{S}^B_n(\clq)$. Then there exist a Hilbert space $\cle$ and an inner function $\Theta \in H^\infty_{\clb(\cle, \cle_*)}(\D^n)$ such that $\clq = \clq_\Theta$. By part (1) of Lemma \ref{wandering}, we have $\clw = \Theta \cle$ and hence
\[
P_\clw = P_{\Theta \cle} = M_\Theta P_\cle M_\Theta^*,
\]
where $P_{\cle}f=f(0)$ for all $f\in H^2_{\cle}(\D^n)$. Then, for each $\tilde{f} \in \clq_\Theta^n$, we know, by \eqref{positive_defect}, that
\[
\la \bm{D}_C^2 \tilde{f}, \tilde{f} \ra = \|P_{\clw}M_{z_1}f_1 + \cdots + P_{\clw}M_{z_n} f_n\|^2 = \|M_{\Theta}P_{\cle} M_{\Theta}^* (M_{z_1} f_1 + \cdots + M_{z_n} f_n)\|^2.
\]
This and the fact that $M_\Theta$ is an isometry imply
\[
\begin{split}
\la \bm{D}_C^2 \tilde{f}, \tilde{f} \ra & \leq \|M_{\Theta}P_{\cle}\|^2 \| M_{\Theta}^*M_{z_1} f_1 + \cdots + M_{\Theta}^*M_{z_n} f_n\|^2
\\
&\leq  \| M_{\Theta}M_{\Theta}^*M_{z_1} f_1 + \cdots + M_{\Theta}M_{\Theta}^*M_{z_n} f_n\|^2
\\
& = \|P_{\cls}M_{z_1} f_1 + \cdots + P_{\cls}M_{z_n} f_n\|^2
\\
&=\la \bm{D}_{c, C}^2 \tilde{f}, \tilde{f} \ra,
\end{split}
\]
proving that $\bm{D}_C^2\leq \bm{D}_{c, C}^2$.
\end{proof}

Following the argument as applied in the first part of the proof of the above theorem, we have $[C_j,C_i^*]=P_{\clq}M_{z_i}^*P_{\cls}M_{z_j}|_{\clq}$ (see Proposition \ref{corss_commutator}) and $D_{C_i}^2=P_{\clq}M_{z_i}^*P_{\cls}M_{z_i}|_{\clq}$ (also see Lemma \ref{lemma: defect Ci}) for all $i \neq j$. Viewing $[C_j,C_i^*] = (P_{\clq}M_{z_i}^*P_{\cls}) (M_{z_j}|_{\clq})$, Douglas' range inclusion theorem implies
\[
\text{ran} [C_j,C_i^*] \subseteq \text{ran} (P_{\clq}M_{z_i}^*P_{\cls}) = \text{ran} (P_{\clq}M_{z_i}^*P_{\cls} M_{z_i} P_{\clq}) = \text{ran} D_{C_i}^2 \subseteq \text{ran} D_{C_i}.
\]
In particular, for $T \in \mathbb{S}_n^{B}(\clh)$, we have that $\overline{ran}[T_j,T_i^*]\subseteq \cld_{T_i}$ for all $i\neq j$. As $\cld_{T_i} \perp \cld_{T_j}$ for sll $i \neq j$, it follows that:
 	
\begin{cor}
Let $T \in \mathbb{S}_n^{B}(\clh)$.  Then
\[
\overline{ran}\left(\bm{D}_{c, T}^2\right) \subseteq\bigoplus_{j=1}^n\cld_{T_j}.
\]
\end{cor}

Applying the same argument as in \eqref{eqn: D D H}, we infer for each $T\in\mathbb{S}_n^B(\clh)$ that
\[
\bigoplus_{j=1}^n \cld_{j,T} \subseteq \clh,
\]
and therefore, as in \eqref{eqn: D D H}, here also we can treat $\overline{ran}\left(\bm{D}_{c, T}^2\right)$ as a closed subspace of $\clh$.

We conclude this subsection by noting that the concept of tuples of hyponormal operators remains ambiguous (however, see \cite{Athavale, Curto}). However, the operator $\bm{D}_{c, T}^2$ for $T \in \mathbb{S}_n^{B}(\clh)$ may serve as a viable candidate for defining the hyponormality of such tuples. This, of course, requires further justification.

\subsection{Revisiting Ahern and Clark}\label{subsect: Ahern and Clark} Now we turn to Beurling quotient modules. We reprove a classical and surprising result found by Ahern and Clark \cite{Analytic Continuation} concerning the structure of submodules within the scalar-valued Hardy space over $\D^n$. We prove the result within the framework of vector-valued Hardy spaces. We present two independent proofs.

 
 
\begin{thm}
Let $\Theta\in H^{\infty}_{\clb(\cle,\cle_*)}(\D^n)$ be a non-constant inner function on $\D^n$, $n\geq 2$. Then
\[
\dim \clq_{\Theta}=\infty.
\]  	
\end{thm}
\begin{proof}[Proof I]
As $\Theta$ is non-constant, there exists $j \in I_n$ such that $M_{z_j}^*\Theta \cle\neq \{0\}$. By Proposition \ref{Structure},  $\cld_{C_j}\neq \{0\}$. Since $n\geq 2$, there exists $i \in I_n$ such that $i\neq j$ and $M_{z_i}|_{\cld_{C_j}}$ is a non-zero pure isometry (see part (3) of Theorem \ref{pure_isometry}). This shows, in particular, that $\dim \cld_{C_j}=\infty$. As $\cld_{C_j}\subseteq \clq_{\Theta}$, the result follows.
\end{proof}
\begin{proof}[Proof II]		
As above, there exists $j \in I_n$ such that $M_{z_j}^*\Theta \cle\neq \{0\}$. By Proposition \ref{Structure}, $\dim \cld_{j,C}\geq 1$. There exists $\eta \in \cle$ such that
\[
h:=M_{z_j}^*\Theta\eta\neq 0.
\]
Pick $i\in I_n$ such that $i \neq j$. Consider the set
\[
F:= \{z_i^kh:k\geq 0  \}.
\]
Let $\{k_1,\dots,k_l\} \subseteq \mathbb{Z}_+$ and $\{c_1,\dots, c_l\}\subseteq \mathbb{C}$ and suppose
\[
c_1z_i^{k_1}h+\dots +c_lz_i^{k_l}h=0.
\]
Then $c_1z_i^{k_1}+\dots +c_lz_i^{k_l}=0$. This implies $c_k=0$ for all $k=1,\dots,l$. Thus, $F$ is an infinite linearly independent subset of $\cld_{C_j}$. So, $\dim \cld_{C_j}=\infty$. As $\cld_{C_j}\subseteq \clq_{\Theta}$, the result follows.
\end{proof}

The above proofs are new even when compared to the scalar case of Ahern and Clark. Moreover, the proofs naturally follow from the techniques introduced in this paper, which are used to derive results of different kinds.

\subsection{Pairs of contractions}\label{subsect: n=2}

We conclude this paper with the representations of characteristic functions for pairs of commuting contractions. This is relevant as for this case, the notions of truncated defect operators and joint commutators are simpler and commonly known. For instance, if $n=2$, then the joint commutators are simply the conventional commutators. In order to be more specific, for each contraction $X \in \clb(\clh)$, we define the \textit{operator Blaschke factor} $b_X: \D \raro \clb(\clh)$ as:
\[
b_X(z) = (1-z X^*)^{-1}(z - X),
\]
for all $z \in \D$.

We pick a pair of commuting contractions $T = (T_1, T_2)$ acting on $\clh$. By the definition of truncated defect operators, we have
\[
\delta_{12}(T) = [T_2, T_1^*].
\]
Moreover, the truncated defect operators are given by
\[
D_{1, T}^2 = D^2_{T_1} - T_2 D^2_{T_1} T_2^*,
\]
and
\[
D_{2, T}^2 = D^2_{T_2} - T_1 D^2_{T_2} T_1^*.
\]
Then the joint defect operator (that is, the defect operator of the second kind) is given by
\[
\bm{D}_T^2:=\begin{bmatrix}
D^2_{T_1} - T_2 D^2_{T_1} T_2^* &  [T_2, T_1^*] \\
[T_1,T_2^*]  & D^2_{T_2} - T_1 D^2_{T_2} T_1^*
\end{bmatrix} \in \clb(\clh^2).
\]
Next, for the pair of commuting contractions $T = (T_1, T_2)$, we define the \textit{joint operator Blaschke factor} $b_{(T_1, T_2)}: \D^2 \raro \clb(\clh)$ as:
\[
b_{(T_1, T_2)}(z) = (I_{\clh} - z_1 T_1^*)^{-1} b_{T_2}(z_2) (I_{\clh} - z_1 T_1^*),
\]
for all $z = (z_1, z_2) \in \D^2$. Finally, assume that $T$ is a Beurling tuple, that is, $T \in \mathbb{S}_2^B(\clh)$. Then the characteristic function of $T$ is given by
\[
\Theta_{T}((z_1, z_2))\bm{D}_T \tilde{h} = D_{T^*}\Big(b_{(T_1, T_2)}(z_1, z_2) h_2 + b_{(T_2, T_1)}(z_2, z_1) h_1\Big),
\]
for all $\tilde{h} = \begin{bmatrix} h_1 & h_2\end{bmatrix}^t \in \clh^2$ and $(z_1, z_2) \in \D^2$, where $D_{T^*} = (\mathbb{S}^{-1}_2(T, T^*))^{\frac{1}{2}}$ is the defect operator of the first kind.

\vspace{0.1in}
	
\noindent\textsf{Acknowledgement:} The first named author extends his gratitude to H{\aa}kan Hedenmalm for several valuable conversations on the topic of this paper. His research is supported by the Verg Foundation. The research of the second named author is partially supported by SG/IITH/F316/2022-23/SG-151, and DST/INSPIRE/04/2021/002101. The research of the third named author is supported in part by TARE (TAR/2022/000063) by SERB, Department of Science \& Technology (DST), Government of India.

\end{document}